\documentclass[12pt]{amsart}
\usepackage{geometry}
\usepackage{amsmath}
\usepackage{amsthm}
\usepackage{bbm}
\usepackage{amsfonts,amssymb,bm}
\usepackage{fancyhdr}
\usepackage{mathrsfs}
\usepackage{appendix}
\usepackage{graphicx}
\usepackage{floatrow}
\usepackage{mathtools}
\usepackage{amsfonts,amssymb}  
\usepackage[all]{xy}
\usepackage{color}

\usepackage{float} 
\usepackage{amsthm}

\newtheorem{thm}{Theorem}[section]
\newtheorem{defi}[thm]{Definition}

\newtheorem{cor}[thm]{Corollary}
\newtheorem{prop}[thm]{Proposition}

\newtheorem{ex}[thm]{Example}
\newtheorem{rmk}[thm]{Remark}
\newtheorem{ques}[thm]{Question}

\usepackage{times}
\usepackage{hyperref}
\def\Hom{{\rm Hom}}
\def\ord{{\rm ord}}
\def\min{{\rm min}}
\def\max{{\rm max}}
\def\inf{{\rm inf}}
\def\sup{{\rm sup}}
\def\lim{{\rm lim}}
\def\limsup{{\rm lim\,sup}}

\def\dif{{\rm d}}

\def\gr{{\rm gr}}
\def\pr{\rm pr}
\def\wt{{\rm wt}}
\def\Rees{{\rm Rees}}
\def\Proj{{\rm Proj}}
\def\Spec{{\rm Spec}}
\def\Diff{{\rm Diff}}

\def\log{{\rm log}}
\def\vol{{\rm vol}}
\def\Val{{\rm Val}}

\def\lct{{\rm lct}}
\def\DH{{\rm DH}}

\def\SL{{\rm SL}}

\def\Bl{{\rm Bl}}

\def\mult{{\rm mult}}

\def\TZ{{\rm TZ}}

\def\BP{\mathbf{P}}

\def\BO{\mathbf{O}}

\def\BH{\mathbf{H}}
\def\BL{\mathbf{L}}

\def\Bv{\mathbf{v}}

\def\tH{\tilde{H}}
\def\tL{\tilde{L}}
\def\tC{\tilde{C}}
\def\tY{\tilde{Y}}
\def\tS{\widetilde{S}}
\def\nv{\mathfrak{v}}
\newcommand{\IA}{{\mathbb A}}

\newcommand{\IC}{{\mathbb C}}

\newcommand{\IG}{{\mathbb G}}

\newcommand{\IN}{{\mathbb N}}
 
\newcommand{\IP}{{\mathbb P}} 
\newcommand{\IQ}{{\mathbb Q}} 
\newcommand{\IR}{{\mathbb R}}

\newcommand{\IT}{{\mathbb T}}

\newcommand{\IZ}{{\mathbb Z}}

\newcommand{\CF}{{\mathcal F}}

\newcommand{\CL}{{\mathcal L}}

\newcommand{\CO}{{\mathcal O}}

\newcommand{\CR}{{\mathcal R}}

\newcommand{\CX}{{\mathcal X}}

\newcommand{\fv}{\mathfrak{v}}

\newcommand{\seq}{\subseteq}
\newcommand{\la}{\langle}
\newcommand{\ra}{\rangle}

\newcommand{\bu}{\bullet}
\newcommand{\lam}{\lambda}
\newcommand{\D}{\Delta}

\newfloatcommand{capbtabbox}{table}[][\FBwidth]

\title{Optimal Degenerations of K-unstable Fano threefolds}

\author{Minghao Miao}
\address{Department of Mathematics, Nanjing University, Nanjing 210093, China}
\curraddr{}
\email{minghao.miao@smail.nju.edu.cn}
\thanks{}
\date{}
\dedicatory{}

\author{Linsheng Wang}
\address{Department of Mathematics, Nanjing University, Nanjing 210008, China}
\curraddr{}
\email{linsheng\_wang@outlook.com}
\thanks{}
\date{}
\dedicatory{}

\begin{document}

\maketitle

\begin{abstract}
We explicitly determine the optimal degenerations of Fano threefolds $X$ in family №2.23 of Mori-Mukai's list as predicted by the Hamilton-Tian conjecture. More precisely, we find a special degeneration $(\mathcal{X}, \xi_0)$ of $X$ such that $(\mathcal{X}_0, \xi_0)$ is weighted K-polystable, which is equivalent to $(\mathcal{X}_0, \xi_0)$ admitting a K\"ahler-Ricci soliton (KRS) by \cite{HL23} and \cite{BLXZ23}. 
Furthermore, we study the moduli spaces of $(\mathcal{X}_0, \xi_0)$. The $\mathbf{H}$-invariant of $X$ divides the natural parameter space into two strata, which leads to different moduli spaces of KRS Fano varieties. We show that one of them is isomorphic to the GIT-moduli space of biconic curves $C\seq \IP^1\times \IP^1$, and the other one is a single point. 
\end{abstract}

\section{Introduction} 
In the original approach to finding K\"ahler-Einstein metrics on Fano manifolds, Tian predicted in \cite[Conjecture 9.1]{Tia97} that a sequence $\omega_t$ of solutions of the normalized K\"ahler-Ricci flow on a Fano manifold $M$ (up to taking subsequence) should converge to $(M_\infty, \omega_\infty)$ in the Cheeger-Gromov-Hausdorff topology, such that $M_\infty$ is smooth outside a subset of Hausdorff codimension $4$ and $\omega_\infty$ is a K\"ahler-Einstein metric or a K\"ahler-Ricci soliton on the smooth part of $M_\infty$. 
This conjecture has been widely studied, and is now solved, see \cite{TZ16, Bam18, CW20, WZ21}. There is an algebraic version of this conjecture which was also solved, see \cite{HL23, BLXZ23}. The limit $M_\infty$ is called the {\it optimal degeneration} of the Fano manifold $M$. 

Although the theoretical aspect of optimal degeneration is well-established, it was almost blank on determining the optimal degeneration explicitly. We have the following question asked by Tian. \vspace{0.3em}

\noindent
{\bf Tian Problem. } How to find the optimal degeneration $M_\infty$ of a K-unstable Fano manifold $M$? \vspace{0.3em}

\noindent
The $\BH$-invariant introduced by \cite{TZZZ13} plays a key role in this problem, which admits a unique minimizer inducing the optimal degeneration of $M$ by \cite{HL20,BLXZ23}. However, the $\BH$-minimizer is difficult to compute explicitly. To the authors' knowledge, the only non-trivial examples come from group compactifications and spherical varieties, see \cite{LL23,TZ22,WZ23}. We will present examples with discrete automorphism groups in this paper.

Another invariant used to characterize K-unstable Fano manifolds $M$ is the so-called stability threshold $\delta(M)$ (or $\delta$-invariant). By \cite{LXZ22}, there exists a prime divisor $E$ over $M$ minimizing $\delta(M)$ and inducing a special degeneration of $M$ to a $\IQ$-Fano variety $M_0$. This should also be considered as the ``most'' degenerating direction of $M$. Hence one may compare the two degenerations induced by the $\delta$-minimizer and $\BH$-minimizer naturally. On the other hand, there is an effective method developed by \cite{AZ22} to compute the $\delta$-minimizer. Therefore, the optimal degeneration problem will be understood if the following question has a positive answer. 

\begin{ques} \label{ques: delta=H}
For a K-unstable Fano manifold $M$, when do the $\delta$-invariant and $\BH$-invariant have the same minimizer (up to rescaling)?
\end{ques}

In this paper, we solve the Tian problem for Fano threefolds in the family №2.23 of Mori-Mukai's list, which splits into two subfamilies, №2.23(a) and №2.23(b). More precisely, we show that the $\delta$-minimizer and the $\BH$-minimizer are the same (up to rescaling) for any Fano threefold $X$ in №2.23(a), but different for those $X$ in №2.23(b). In particular, for any $X$ in №2.23(a), we determine the optimal degeneration $(X_0,\xi_0)$. On the other hand, even more non-trivial is that, we find the optimal degenerations of any Fano threefolds in №2.23(b). 

It was proved by the first author and Tian (\cite{MT22}) that the Fano threefolds in family №2.23 are the examples of lowest dimension on which the normalized K\"ahler-Ricci flow develops the type II singularity. 
Our result gives an affirmative answer to the question posed in \cite[Remark 3.1]{MT22} regarding the explicit description of their optimal degeneration.

Recall that any Fano threefold $X$ of family №2.23 is obtained by blowing up the quadric threefold $Q\seq \IP^4$ along a smooth elliptic curve $C$, which is the complete intersection of a hyperplane section $H\in |\CO_Q(1)|$ and a quadric section $Q'\in |\CO_Q(2)|$. The family consists of two subfamilies: №2.23(a) and №2.23(b) determined by $H\cong \IP^1\times \IP^1$ and $H\cong \IP(1,1,2)$ respectively. We denote by $\tH$ the strict transform of $H$ on $X$. 

\begin{thm}
For any $X$ in №2.23, we have $\delta(X)=\frac{12}{13}$, which is minimized by $\ord_{\tH}$. 
\end{thm}

The divisorial valuation $\ord_{\tH}$ induces a special degeneration $\pi:(\CX,\xi) \to \IA^1$ of $X$ to $X_0$, where $X_0=\Bl_{C}Q_0$ and $Q_0$ is the cone over $(H, \CO_Q(1)|_H)$. 
Let $\xi_0\in N_\IR\cong\IR$ be the soliton candidate of $X_0$ (see the definition in Section \ref{sec:weighted}). Then $\xi_0=a_0 \cdot \xi$ for some $a_0>0$, and $v_0=a_0 \cdot \ord_{\tH}$ induces the special degeneration $(\CX,\xi_0)$. 


\begin{thm} 
\label{Theorem: Intro, №2.23(a), (X_0, xi_0) K-polystable}
For any $X$ in №2.23(a), the pair $(X_0, \xi_0)$ is weighted K-polystable. 
\end{thm}

Combining with Theorem \ref{Theorem: H minimizer = central fiber K-ss}, we get the optimal degeneration of $X$ in №2.23(a). 

\begin{cor}
The divisorial valuation $v_0 = a_0 \cdot \ord_{\tH}$ minimizes the $\BH$-functional of $X$ in №2.23(a). Hence $(X_0,\xi_0)$ is the optimal degeneration of $X$. 
\end{cor}

Moreover, we have a description of the moduli space of the KRS Fano varieties $(X_0, \xi_0)$. 

\begin{thm}
\label{Theorem: Intro, moduli of optimal degeneration of 2.23(a)}
Let $Q_0=C(\IP^1\times\IP^1, \CO(1,1))$ be a cone and $H\cong \IP^1\times \IP^1$ be the section at infinity. Let $X_0$ be the blowup of $Q_0$ along $C=Q'_0|_{H}$ where $Q'_0\in|\CO_{Q_0}(2)|$. Denote by $\xi_0$ the soliton candidate of $X_0$. 
Then $(X_0, \xi_0)$ is weighted K-semistable (weighted K-polystable) if and only if $C \seq H$ is GIT-semistable (GIT-stable or polystable) as a $(2,2)$-divisor in $\IP^1 \times \IP^1$. 
\end{thm}


However, similar results do not hold for Fano threefolds in №2.23(b). 

\begin{prop}
For any $X$ in №2.23(b), the pair $(X_0, \xi_0)$ is weighted K-unstable. 
\end{prop}


We are now ready to describe the optimal degeneration of Fano threefolds in №2.23(b). Assume that $X=\Bl_C Q$ is in №2.23(b), where $C= Q'|_H$ does not pass through the vertex $o\in X$ of $H\cong \IP(1,1,2)$. Let $E_o$ be the exceptional divisor of the blowup $\Bl_o X\to X$, which is a toric divisor over $Q$ with respect to the $\IT$-action (since $C$ is not $\IT$-invariant, the $\IT$-action could not be lifted to $X$). Then the divisorial valuation $\ord_{E_o}$ induces a special degeneration $(\CX, \eta)$ of $X$ with central fiber $X_0 = \Bl_{C_0} Q$, where $C_0=2H'|_H$ is a non-reduced plane conic and $H'\in|\CO_Q(1)|$ is another $\IT$-invariant hyperplane section not containing $o$. The Fano variety $X_0$ has $A_1$-singularity along a smooth plane conic curve and is smooth outside the curve. 
We denote by $\eta_0$ the soliton candidate of $X_0$ with respect to the $\IG_m$-action induced by $(\CX,\eta)$. Then $\eta_0 = b_0\cdot \eta$ for some $b_0 > 0$, and $v_0=b_0 \cdot \ord_{E_o}$ induces the special degeneration $(\CX,\eta_0)$. 

\begin{thm} 
\label{Theorem: Intro, №2.23(b), (X_0, eta_0) K-polystable}
For any $X$ in №2.23(b), the pair $(X_0, \eta_0)$ is weighted K-polystable. 
\end{thm}

\begin{cor}
The divisorial valuation $v_0 = b_0 \cdot \ord_{E_o}$ minimizes the $\BH$-functional of $X$ in №2.23(b). Hence $(X_0,\eta_0)$ is the optimal degeneration of $X$. 
\end{cor}

It was proved by \cite[Theorem 6.4]{BLXZ23} that the $\BH$-invariant is constructible and lower semi-continuous on a locally stable family of log Fano pairs. Our examples illustrate that the $\BH$-invariant can indeed be strictly lower semi-continuous, which is a new phenomenon. In fact, according to \cite{MT22}, we know that family №2.23 has a two-dimensional parameter space, where family №2.23(b) constitutes its one-dimensional closed subset $Z$, and the remaining open set $U$ parametrizes family №2.23(a). Our proof shows that the $\BH$-invariant $h(X_p)$ is constant when $p\in Z$ or $p\in U$. However, if $p$ transits from $U$ to $Z$, the value of $h(X_p)$ will decrease abruptly. This is a manifestation of the lower semi-continuity of $\BH$-invariant.

Recently, \cite{ACC+} are working on a project which aims to find all K-polystable smooth Fano threefolds in each family of Mori-Mukai's list. It is natural for us to ask for more under the framework of the Tian problem. If $X$ is K-polystable, then the optimal degeneration is trivial. If $X$ is strictly K-semistable, then there exists a unique K-polystable degeneration by \cite{LWX21} as its optimal degeneration. If $X$ is K-unstable but admits K\"ahler-Ricci soliton, then by \cite{TZ07,TZ15,TZZZ13}, the product test configuration induced by the soliton vector field gives the optimal degeneration. If $X$ is K-unstable and does not admit K\"ahler-Ricci soliton, then we need to determine the optimal degeneration $X_{\infty}$ explicitly. We expect the method developed in this paper will help to give a complete answer to Tian Problem for smooth Fano threefolds in the future. 


\noindent {\bf Acknowledgments}. The authors would like to thank their advisor, Gang Tian, for suggesting the problem of finding optimal degenerations of K-unstable Fano manifolds. The authors thank Yuchen Liu for pointing out the $\ord_{\tH}$-degeneration, and thank Xiaohua Zhu for his suggestion on comparing the $\delta$-minimizer and the $\BH$-minimizer. The authors also thank Jiyuan Han, Chi Li and Feng Wang for helpful discussions. LW is partially supported by the NKRD Program of China (\#2025YFA1018100) and (\#2023YFA1010600). MM is partially supported by NSFC Grant 125B2003. The authors are partially supported by the NKRD Program of China (\#2020YFA0712800).

\section{Preliminaries}

We work over the field of complex numbers $\IC$. A {\it log Fano pair} $(X,\D)$ consists of a projective klt pair $(X,\D)$ with $-(K_X+\D)$ ample $\IQ$-Cartier. Let $l_0\in\IN$ be an integer such that $l_0(K_X+\D)$ is Cartier. We denote by $R=R(X,\D)=\oplus_{m\in l_0\IN} R_m$ the anti-canonical ring of $(X,\D)$, where $R_m=H^0(X,-m(K_X+\D))$. 

\subsection{$\IR$-test configurations} 
We briefly recall the notion of $\IR$-TC in this subsection, see for example \cite[Section 2.2]{HL20}. 
A {\it filtration} $\CF$ on the graded algebra $R = \oplus_m R_m$ is a collection of subspaces $\CF^\lam R_m \subseteq R_m$ for each $\lam \in \IR$ and $m \in \IN$ such that
\begin{itemize}
\item {\it Decreasing.} $\CF^\lam R_m \supseteq \CF^{\lam'}R_m $ for  $\lam \le \lam'$; 
\item {\it Left-continuous.} $\CF^\lam R_m=\CF^{\lam-\epsilon}R_m$ for $0<\epsilon \ll 1$; 
\item {\it Bounded.} $\CF^\lam R_m = R_m$ for $\lam \ll 0$ and $\CF^\lam R_m = 0$ for $\lam \gg 0$; 
\item {\it Multiplicative.} $\CF^\lam R_m \cdot \CF^{\lam'}R_{m'} \subseteq \CF^{\lam+\lam'}R_{m+m'}$. 
\end{itemize}
It is called {\it linearly bounded} if there is a constant $C>0$ such that $\CF^{mC}R_m=0$ and $\CF^{-mC}R_m=R_m$ for all $m\in\IN$. The graded linear series $\{\CF^{(t)}R_{\bu}\}$ is given by $(\CF^{(t)}R_{\bu})_m\coloneqq \CF^{mt}R_m$. We will always assume that filtration is linearly bounded in this paper. The {\it associated graded algebra} of $\CF$ is defined by $\gr_\CF R:= \oplus_m \gr_\CF R_m$ where $\gr_\CF R_m = \oplus_{\lam \in \IR} \CF^\lam R_m / \CF^{>\lam}R_m$. 

Let $a\in\IR_{>0}, b\in \IR$, we define the $a$-{\it rescaling} and $b$-{\it shift} of $\CF$ by
$$(a\CF)^\lam R_m := \CF^{\lam/a}R_m, \,\,
  \CF(b)^\lam R_m := \CF^{\lam-bm} R_m, $$
and we also denote by $a\CF(b):=(a\CF)(b)$, that is $\big(a\CF(b)\big)^\lam R_m = \CF^{\frac{\lam - bm}{a}} R_m$. 

For any valuation $v$ over $X$, we denote by 
$$\CF_v^\lam R_m=\{s\in R_m\mid v(s) \ge \lam\} $$
the corresponding filtration. It is linearly bounded if and only if $A_{X,\D}(v)<+\infty$, see \cite{JM12}. Note that the rescaling of a valuation and the corresponding filtration are compatible, that is, $\CF_{a\cdot v} = a\CF_v$. 

\begin{defi} \rm
A linearly bounded filtration $\CF$ on $R_\bu=R(X,\D)$ is called an $\IR$-{\it test configuration} ($\IR$-TC) of $(X,\D)$ if its associated graded algebra $\gr_\CF R$ is finitely generated. 
\end{defi}

The terminology comes from the following observations. The {\it Rees algebra} of $\CF$ is defined by 
$$\Rees_\CF R
:=\bigoplus_{m}\bigoplus_{\lam\in \Gamma_m} 
t^{-\lam} \cdot \CF^{\lam} R_m, $$
where $\Gamma_m=\Gamma_m(\CF)=\{\lam\in\IR\mid \CF^\lam R_m / \CF^{>\lam}R_m \ne 0\} \seq \IR$ is the subset of all real numbers $\lam$ that the filtration jumps. 
It's clear that $\gr_\CF R$ is finitely generated if and only if $\Rees_\CF R$ is finitely generated. In this case, the submonoid $\Gamma=\cup_m\Gamma_m\seq \IR$ is of rank $r<\infty$. We say that $\CF$ is of {\it rational rank} $r$. There exists a unique basis $\eta_1,\cdots, \eta_r \in \IR$ of the $\IQ$-vector space $\Gamma_\IQ\seq \IR$, such that 
$$\Gamma\seq \sum_{1\le i\le r}\IN \cdot \eta_i, \quad
\Gamma\nsubseteq\sum_{1\le i\le r}\IN \cdot a_i\eta_i, $$ 
for any $(1,\cdots, 1)\ne(a_1,\cdots, a_r) \in \IN^r$. Let $t_i:=t^{\eta_i}$, then we see that $\Rees_\CF R$ is a finitely generated $A=\IC[t_1, \cdots, t_r]$-algebra, and we get a morphism 
\begin{eqnarray}
\label{Morphism: multi-degeneration}
\pi: \CX=\Proj_A(\Rees_\CF R) \to \Spec A = \IA^r, 
\end{eqnarray}
with general fiber $\CX_{t}=X$ and special fiber $\CX_0=\Proj_\IC(\gr_\CF R)$. We denote by $\CL=\CO_{\CX/\IA^r}(1)$. Then $\CL|_{X\times\IG_m^r} \cong -\pr^*( K_X +\D)$. Note that the natural decomposition of $\Rees_\CF R$ is just the $\IT=\IG_m^r$-weight decomposition, which induces a $\IT=\IG_m^r$-action on $(\CX,\D_\CX)$ making $\pi$ equivariant. We shall define $\D_{\CX}$ to be the closure of $\D\times \IG_m^r$ in $\CX$. Then $(\CX,\D_\CX)$ is a $\IT=\IG_m^r$-equivariant family that degenerates $(X,\D)$ to $(\CX_0, \D_{\CX,0})$. 
This generalizes the notation of test configurations and the terminology of $\IR$-TC makes sense. 

Let $N=\IZ^r=\oplus_{1\le i\le r} \IZ \cdot e_i$, $M=N^\vee$ and $\eta=\sum_i \eta_i e_i \in N_\IR$. Write $\Rees_\CF R = \oplus_m \CR_m$, where $\CR_m=\oplus_{\lam\in \Gamma_m} t^{-\lam} \CF^\lam R_m$. Recall that any $\lam \in \Gamma_m$ is of the form $\lam=\sum_i \alpha_i \eta_i$ for some $\alpha_i\in \IN$. Then $\lam = \la \alpha, \eta\ra$ where $\alpha=\sum_i \alpha_i e^\vee_i \in M$. Hence we have $\IT=\IG_m^r$-weight decomposition 
\begin{eqnarray*}
\CR_m =\bigoplus_{\alpha \in M} \CR_{m, \alpha}, 
&\quad&
\CR_{m, \alpha} = t^{-\la \alpha, \eta\ra} \CF^{\la \alpha, \eta\ra} R_m, \\
\gr_\CF R_m =\bigoplus_{\alpha \in M} \gr_\CF R_{m, \alpha}, 
&\quad&
\gr_\CF R_{m, \alpha} =  
\CF^{\la \alpha, \eta\ra} R_m / \CF^{>\la \alpha, \eta\ra} R_m. 
\end{eqnarray*}
The vector field $\eta \in N_\IR$ determined by the $\IR$-TC $\CF$ induces a (product type) filtration $\CF'=\CF_{\wt_\eta}$ on $\Rees_\CF R$ and $\gr_\CF R$, 
\begin{eqnarray*}
\CF_{\wt_\eta}^\lam\CR_m =\bigoplus_{\la \alpha, \eta\ra \ge \lam} \CR_{m, \alpha}, 
\quad
\CF_{\wt_\eta}^\lam \, \gr_\CF R_m =\bigoplus_{\la \alpha, \eta\ra \ge \lam} \gr_\CF R_{m, \alpha}. 
\end{eqnarray*}
Hence the $\IR$-TC $\CF$ of $(X,\D)$ is also denoted by $(\CX,\D_\CX,\CL,\eta)$, and $\eta$ is called the {\it degeneration vector field} corresponding to $\CF$. 
If $\CF$ is replaced by a rescaling $a\CF$, then $(\CX,\D_\CX,\CL)$ are not changed, but the degeneration vector field $\eta$ will be replaced by $a\cdot \eta$. 

\begin{rmk}\rm
If $\CF=\CF_v$ is induced by some valuation $v$ over $X$, then the central fiber $\CX_0$ of the $\IR$-TC is integral. Indeed, this is equivalent that $\gr_{\CF} R$ is integral. For any non-zero $\bar{s}\in \gr_{\CF}^\lam R_m, \bar{s}'\in\gr_{\CF}^{\lam'} R_{m'}$, we see that their lifting $s,s'\in R$ satisfy $v(s)=\lam, v(s')=\lam'$ respectively, hence $v(ss')=\lam+\lam'$. So $ss' \notin \CF^{>\lam+\lam'}R_{m+m'}$, that is, $\overline{ss'}\ne 0$ in $\gr_\CF R$. 
\end{rmk}

\subsection{$\BH$-invariants} 

We recall the $\BH$-invariant introduced by \cite{TZZZ13} and reformulated by \cite{DS20,HL20} in this subsection, which is essential in the study of optimal degeneration of Fano manifolds. In the following, we assume that $(X,\D)$ is a log Fano pair, and $R=R(X,\D)$ is the associated anti-canonical ring.

\begin{defi}\rm
Let $\CF$ be a linearly bounded filtration on $R$ and $m\in l_0\IN$. For any $s\in R_m$, we set $\ord_\CF(s)=\max\{\lam\mid s\in\CF^{\lam}R_m\}$. A basis $\{s_i\}$ of $R_m$ is called {\it compatible} with $\CF$ if $\CF^\lam R_m$ is generated by $\{s_i\mid \ord_\CF(s_i)\ge \lam \}$ for any $\lam\in \IR$. For such a basis, it determines the set of real numbers $\Gamma_m=\{\lam_i^{(m)}=\ord_\CF(s_i)\}$ where the filtration jumps. The collection $\{\lam_i^{(m)}\}$ is not dependent on the choice of the compatible basis and is called the {\it successive minimum} of $\CF$. We denote by $\lam^{(m)}_\max$ and $\lam^{(m)}_\min$ the maximum and the minimum number, respectively. 
Then $\lam^{(m)}_\max = \max\{\lam\in \IR\mid \CF^{\lam}R_m \ne 0\}$. We denote by 
\begin{eqnarray*} 
\lam_\max 
= \lam_\max(\CF) 
:= \mathop{\sup}_{m\in \IN} \frac{\lam^{(m)}_\max}{m} 
= \mathop{\lim}_{m\rightarrow \infty} \frac{\lam^{(m)}_\max}{m}. 
\end{eqnarray*}
\end{defi}

\begin{defi}[Log canonical slopes] \rm
Let $L=-(K_X+\D)$. 
For any linearly bounded filtration $\CF$ on $R$, the {\it base ideal sequence} $I^{(t)}_\bu = \{I_{m,mt}\}_{m\in l_0\IN}$ of $\CF$ is defined by 
\begin{eqnarray*}
I_{m,mt}
&:=&{\rm im}
\Big(\CF^{mt}H^0(X,mL)\otimes \CO(-mL)\to \CO\Big), 
\end{eqnarray*}
for any $m\in l_0\IN$ and $t\in\IR$. The {\it log canonical slope} of $\CF$ is defined by
\begin{eqnarray*}
\mu(\CF)
&:=& \mu_{X,\D}(\CF)
\,\,\,:=\,\,\,
\sup\Big\{
t: \lct(X,\D;I^{(t)}_\bu)\ge 1
\Big\}. 
\end{eqnarray*}
Note that $I^{(t)}_\bu=0$ (hence $\lct(X,\D;I^{(t)}_\bu)=0$) when $t>\lam_\max$. We have $\mu(\CF)\le \lam_\max$. By \cite[Proposition 4.2]{XZ19}, we know that $\mu(\CF_v) = A_{X,\D}(v)$ for any weakly special valuation $v$ over $X$. 
\end{defi}

\begin{defi} \rm
Let $\CF$ be a linearly bounded filtration on $R=R(X,\D)$. 
We define 
\begin{eqnarray*} 
S_m(\CF) \,\,\,=\,\,\,
\sum_{\lam}\frac{\lam}{m} \cdot 
\frac{\dim \gr_\CF^{\lam} R_m}{\dim R_m}, \qquad
\tS_m(\CF) \,\,\,=\,\,\,
- \,\log\Big(
\sum_{\lam} e^{-\frac{\lam}{m}} \cdot 
\frac{\dim \gr_\CF^{\lam} R_m}{\dim R_m}
\Big). 
\end{eqnarray*} 
Their limits are the so-called $S$-{\it invariant} and $\tS$-{\it invariant}
\begin{eqnarray*} 
S(\CF) = \mathop{\lim}_{m\to \infty} S_m(\CF), \qquad
\tS(\CF) = \mathop{\lim}_{m\to \infty}\tS_m(\CF). 
\end{eqnarray*} 
\end{defi}

\begin{rmk} \rm
Using DH measures or Okounkov bodies, one can reformulate these invariants. Let $\DH_{\CF,m} = \frac{1}{N_m}\sum_\lam \delta_{\frac{\lam}{m}} N_{m}^{\lambda}$ be the discrete DH measures on $\IR$ determined by $\CF$, and $\DH_\CF$ be the associated weak limit \cite[Lemma 2.8]{BJ20}. Then we have 
\begin{eqnarray*} 
S_m(\CF) \,\,\,=\,\,\,
\int_\IR t \cdot \DH_{\CF,m}(\dif t), 
&\,\,&
\tS_m(\CF) \,\,\,=\,\,\,
-\,\log\Big(
\int_\IR e^{-t} \cdot \DH_{\CF,m}(\dif t)
\Big), \\
S(\CF) \,\,\,=\,\,\,
\int_\IR t \cdot \DH_{\CF}(\dif t), 
\,\,\,\,\,&\,\,& \,\,\,\,
\tS(\CF) \,\,\,=\,\,\,
-\,\log\Big(
\int_\IR e^{-t} \cdot \DH_{\CF}(\dif t)
\Big). 
\end{eqnarray*} 
\end{rmk}

\begin{defi}\rm
For any linearly bounded filtration $\CF$ on $R$, we define the $\BH$-{\it functional} as 
\begin{eqnarray*} 
\BH(\CF) \,\,\,:=\,\,\, \mu(\CF) - \tS(\CF). 
\end{eqnarray*} 
Then the $\BH$-{\it invariant} of the log Fano pair $(X,\D)$ is defined by 
\begin{eqnarray*} 
h(X,\D) \,\,\,:=\,\,\, \mathop{\inf}_\CF\BH(\CF), 
\end{eqnarray*} 
where the infimum runs over all the linearly bounded filtrations $\CF$ on $R$. 
\end{defi}

\begin{rmk}\rm
The original definition of $\BH$-functional is $\BH(\CF)=\BL(\CF)-\tS(\CF)$, where $\BL$ is the non-archimedean L-functional, see \cite{BHJ17}. It was proved in \cite[Theorem 3.55]{Xu23} that $\BL(\CF)=\mu(\CF)$ for any linearly bounded filtration $\CF$. Hence the above definition is equivalent to the original one. 
\end{rmk}

\begin{defi} \rm
The $\delta$-{\it invariant} of a log Fano pair $(X, \D)$ is defined by 
$$\delta(X,\D) \,\,\,=\,\,\, 
\inf_v \frac{A_{X,\D}(v)}{S(v)}, $$
where the infimum runs over all the valuations $v$ over $X$, and $S(v)=S(\CF_v)=S(-K_X-\D;\CF_v)$.

\subsection{Weighted stability threshold} \label{sec:weighted}
Let $(X, \D)$ be a log Fano pair with a $\IT$-action. Then there is a canonical lifting of the $\IT$-action on the anti-canonical ring $R_\bu=R(X,\D)$. We denote the cocharacter lattice by $N= \Hom(\IG_m,\IT)$ and the character lattice by $M=N^{\vee}$. The corresponding $\IR$-vector space is $N_{\IR}=N \otimes_{\IZ} \IR,\, M_{\IR}=M\otimes_{\IZ}\IR$. We have the following weight decomposition $R_m=\oplus_{\alpha\in M}R_{m,\alpha}$, where the weight space is given by
$$
R_{m,\alpha}=\{s\in R_m \mid \forall \,t\in \IT,\, t\cdot s = \alpha(t)s\}.
$$
Denote
\begin{eqnarray*}
\BP_m(R_\bu)
\coloneqq \{\alpha\in M\mid R_{m,\alpha}\ne0\}, \,
\BP(R_\bu)
\coloneqq\overline{\bigcup_{\it{m}} \frac{1}{\it{m}} \BP_{\it{m}}(R_\bu)}. 
\end{eqnarray*}
The closed convex body $\BP=\BP(R_\bu)\seq M_\IR$ is called the {\it moment polytope} of the $\IT$-action on $R=R_\bu$. We denote $N_{m,\alpha}=\dim R_{m,\alpha}$. And we define the following discrete measure on $M_{\IR}$ as
\begin{eqnarray*} 
\DH_{\BP,m}
&\coloneqq&\frac{1}{m^n}\sum_{\alpha\in \BP_m}  
N_{m,\alpha} \cdot
\delta_\frac{\alpha}{m}.
\end{eqnarray*}
By \cite[Definition-Theorem 2.31]{Xu23}, the $\DH_{\BP,m}$ weakly converges to a measure $\DH_{\BP}$ as $m\rightarrow +\infty$, which is called {\it $\IT$-equivariant Duistermaat-Heckman measure}. 
\begin{defi}
    We say $(X,\D,\xi_0)$ is a log Fano triple if $(X,\D)$ is a log Fano pair with a $\IT=(\IG_m)^r$-action and $\BP=\BP(R_{\bu})\subseteq M_{\IR}=\IR^r$ is the moment polytope of the $\IT$-action on $R_{\bu}$ and $\xi_0\in N_{\IR}=\IR^r$. We always assume the weight function is given by $g=g_{\xi_0}\colon \BP\rightarrow \IR_{>0}, \alpha \mapsto e^{-\langle \xi_0,\alpha \rangle}$. For simplicity, we denote $\Bv^g=\int_{\BP}g(\alpha)\DH_{\BP}(\dif \alpha)$.
\end{defi}
Fix $\xi\in N_{\IR}$ and $g=g_{\xi}$. We define the {\it Tian-Zhu invariant} as follows
$$\TZ_g(\eta)=\frac{1}{\Bv^g}
\int_\BP \la\alpha, \eta\ra g(\alpha) \DH_\BP(\dif \alpha), \,\,\eta\in N_\IR,$$
whose analytic definition first appeared in the work of \cite{TZ02} (see \cite{TZ25} for a more general setting). By \cite[Lemma 2.35]{HL23}, the $\BH$-functional restricted to the weight valuations $\BH: N_\IR \to \IR, \eta \mapsto \BH(\wt_{\xi+\eta})$ is a strictly convex function and admits a unique minimizer $\xi_0\in N_{\IR}$. The directional derivative of $\BH$-functional along $\eta$-direction is given by $\frac{\dif}{\dif t}\big{|}_{t=0} \BH(\wt_{\xi+t\eta})=-\TZ_{g_{\xi}}(\eta)$. Thus, for the critical point $\xi_0$, we have $\TZ_{g_{\xi_0}}(\eta)=0$ for all $\eta\in N_\IR$. The vector field $\xi_0$ is called the {\it soliton candidate} of $(X,\D)$ with respect to the $\IT$-action. In the following, we always assume the vector field $\xi_0$ appearing in the weight function $g=g_{\xi_0}$ is given by the soliton candidate.

\end{defi}
Let $\CF=\CF R_{\bu}$ be a filtration on $R$. We say $\CF$ is $\IT$-equivariant if $\CF^x R_{m}$ is a $\IT$-invariant subspace of $R_m$ for every $x\in \IR$. Then $(\CF^{x}R_{m})_{\alpha}=\{s\in \CF^{x}R_{m}\mid \forall \, t\in \IT, t\cdot s = \alpha(t)s\}=\CF^x R_m \cap R_{m,\alpha}$.  Following \cite[Section 5.6]{HL23}, we define the $g$-volume of graded linear series $\{\CF^{(t)}R_{\bu}\}$ as 
\begin{eqnarray*}
    \vol^g(\CF^{(t)}R_{\bu}) = \mathop{\limsup}_{m\rightarrow \infty} \sum_{\alpha \in \BP_m} g\left(\frac{\alpha}{m}\right)\frac{\dim \left((\CF^{mt}R_{m})_{\alpha}\right)}{m^n/n!}.
\end{eqnarray*}

In fact the limit exists and we define the Duistermaat-Heckman ($\DH$) measure 
\begin{eqnarray*}
    \DH_{\CF}^g(\dif t)= - \frac{1}{n!}\dif\,\vol^g(\CF^{(t)}R_{\bu})
\end{eqnarray*}
as a derivative distribution. We define the {\it $g$-weighted expected vanishing order} of $\CF$ by 
\begin{eqnarray*} 
S^g(R_\bu; \CF) 
&\coloneqq& 
\frac{1}{ \int_\IR \DH^g_\CF} \int_\IR t \cdot \DH^g_\CF(\dif t). 
\end{eqnarray*} 

Recall that a faithful valuation $\nv$ is adapted to the torus action if for any $f\in \IC(X_0)_{\alpha}$, we have $\nv(f)=(\alpha, {\nv}^{r+1}(f),\cdots,{\nv}^n(f))\in \IZ^r \times \IZ^{n-r}$ (see \cite[Definition 2.21]{HL23}). And an Okounkov body is called compatible with the $\IG_m$-action if the associated faithful valuation with respect to the admissible flag is adapted to the $\IG_m$-action. If $\BO\seq\IR^n$ is an Okounkov body of $-(K_X+\D)$ which is compatible with the $\IT$-action, then there is a linear projection $p:\IR^n\to M_\IR$ mapping $\BO$ onto $\BP$. For any $\IT$-invarant linearly bounded filtration $\CF$ on $R=R_\bu$, let $G^\CF$ be the concave transform of the filtration $\CF$. We have $\BO(\CF^{(t)}R_{\bu}) = \{x\in \BO(R_{\bu})\mid G^{\CF}(x)\geq t\}$.

If $\CF=\CF_E$ is induced by a toric divisor $E$, then $G^\CF$ is invariant on each fiber of $p$, hence descends to a function on $\BP$ (still denoted by $G^\CF$). Then, 
\begin{eqnarray} \label{eq:S^g}
S^g(R_\bu; \CF) 
&=& 
\frac{1}{\Bv^g} 
\int_\BP G^\CF 
\cdot g(\alpha) \DH_{\BP}(\dif \alpha). 
\end{eqnarray} 
We further define the weighted stability threshold, which first appeared in the work of \cite{RTZ21, BLXZ23}.

\begin{defi}
    Let $(X,\D,\xi)$ be a log Fano triple and for any $\IT$-invariant subvariety $Z\seq X$. We define the local weighted stability thresholds (resp. weighted stability thresholds) as 
\begin{eqnarray*}
\delta^g_{Z, \IT}(X,\D;R_\bu) 
= \mathop{\inf}_{v, C_X(v)\supseteq Z} \frac{A_{X,\D}(v)}{S^g(R_\bu; v)}, \, {\it resp.}\,\,\,\, \delta^g_{\IT}(X,\D;R_\bu) 
= \mathop{\inf}_{v\in \Val_X^{\IT,\circ}} \frac{A_{X,\D}(v)}{S^g(R_\bu; v)},
\end{eqnarray*}
where the infimum runs over all the valuations $v\in \Val^{\IT,\circ}_X$ satisfying $C_X(v)\supseteq Z$ (resp. the infimum runs over all the valuations $v\in \Val^{\IT,\circ}_X$).
\end{defi}
\begin{defi} \rm
\label{Definition: T-equivariantly weighted K-semistable}
A log Fano triple $(X,\D, \xi_0)$ with a $\IT$-action is called {\it weighted K-semistable} if $A_{X,\D}(v) - S^g(R_\bu;v) \ge 0$ for any $v\in\Val_{X}^\IT$, namely, it is equivalent to $\delta_{\IT}^g(X,\D;R_{\bu})\geq 1$. It is {\it weighted K-polystable} if it is weighted K-semistable, and $A_{X,\D}(v) - S^g(R_\bu;v) = 0$ implies that $v=\wt_\xi$ for some $\xi\in N_\IR$. 
\end{defi}
\begin{rmk}
    The original definition of weighted K-semistable(K-polystable) using test configuration (see for example \cite[Definition 2.42]{HL20}) is equivalent to Definition \ref{Definition: T-equivariantly weighted K-semistable} by the valuative criterion proved in \cite{HL23} and the result by \cite[Proposition 4.14]{BLXZ23}.
\end{rmk}
\subsection{Optimal degenerations}


By \cite[Theorem 3.7]{BLXZ23}, the $\BH$-functional is strictly convex along geodesics. Combining with \cite[Corollary 4.9]{HL20}, we see that the $\BH$-invariant admits a unique minimizer $v_0$, which is a quasi-monomial valuation over $X$. 

Furthermore, the associated graded algebra $\gr_{v_0}R$ is finitely generated by \cite[Corollary 5.7]{BLXZ23}. Hence the filtration $\CF=\CF_{v_0}$ induced by $v_0$ is an $\IR$-TC of $(X,\D)$, which leads to a special degeneration $(\CX,\D_\CX, \CL,\eta)$ of $(X,\D)$. The log Fano triple $(\CX_0,\D_{\CX,0},\eta_0)$ is weighted K-semistable by \cite[Theorem 5.3]{HL20}. We say that the special degeneration $(\CX_0,\D_{\CX,0},\eta_0)$ induced by the $\BH$-minimizer is the {\it optimal degeneration} of $(X,\D)$. 




\begin{thm}\cite[Theorem 5.3]{HL20}. 
\label{Theorem: H minimizer = central fiber K-ss}
Let $(X,\D)$ be a log Fano pair. Let $v_0$ be a quasi-monomial valuation over $X$ whose associated graded algebra is finitely generated, and $(\CX,\D_\CX,\CL,\eta_0)$ be the associated $\IR$-TC induced by $v_0$. Then
$v_0$ is the $\BH$-minimizer if and only if $(\CX_0, \D_{\CX, 0},\eta_0)$ is weighted K-semistable. 
\end{thm}

\begin{rmk} \rm
Moreover, if $v_0$ is of rational rank one, that is, $v_0=a\cdot \ord_E$ for some $a\in \IR_{>0}$ and prime divisor $E$ over $X$, and $\ord_E$ induces a special degeneration $(\CX,\D_\CX,\CL,\eta)$ of $(X,\D)$. Then the $\IR$-TC induced by $v_0$ is $(\CX,\D_\CX,\CL,a\cdot\eta)$. 
\end{rmk}

\subsection{Weighted Abban-Zhuang Estimate}

We recall the criterion of weighted K-stability developed in \cite{MW23}, which is the key ingredient in finding optimal degenerations. For the notions of multi-graded linear series $V_\bu$ and its associated $S^g$-invariant, we refer the reader to \cite[Section 2.6 \& 4.1]{MW23}. 

Let $(X,\D)$ be a pair with a $\IT$-action, $V_{\bu}$ be a multi-graded linear series on $X$ admitting the $\IT$-action, and $F$ be a plt-type divisor over $X$. Assume that $F$ induces a $\IT$-invariant filtration $\CF_F$ on $V_\bu$. 
We denote by $W_{\bu}$ the refinement of $V_\bu$ by $F$. 
Let $\pi: Y \to X$ be the associated plt-type blowup. We denote by $\D_Y$ the strict transform of $\D$ on $Y$. Then we have 
\begin{eqnarray*}
K_Y+\D_Y+(1-A_{X,\D}(F))F=\pi^*(K_X+\D).
\end{eqnarray*}
We denote by $\D_F$ the difference of $\D_Y$ on $F$, that is
\begin{eqnarray*}
(K_Y+\D_Y+F)|_F=K_F+\D_F. 
\end{eqnarray*}

\begin{thm} \label{Theorem: weighted AZ}
Let $Z \subseteq X$ be a $\IT$-invariant subvariety contained in $C_X(F)$. Then we have 
\begin{eqnarray} \label{ineq.1}
\delta^{g}_{Z, \IT}(X,\D;V_\bu) \ge \min\Big\{\frac{A_{X,\D}(F)}{S^{g}(V_\bu; F)},\,\, \mathop{\inf}_{\pi(Z')=Z} \delta^{g}_{Z',\IT}(F,\D_F;W_\bu) \Big\}. 
\end{eqnarray}
\end{thm}
\begin{ex}\cite[Example 4.1]{MW23} \label{prop:SgW}
Let $V_{\bu}$ be a $\IT$-invariant $\IN\times \IN^l$-graded linear series on $X$. Let $\CF$ be a $\IT$-invariant linearly bounded filtration on $V_{\bu}$. Then,
\begin{eqnarray} 
\label{Formula: S^g 5}
S^g(V_\bu; \CF) 
&=&
\frac{1}{\Bv^g} 
\int_\BP S(V_{(1,\alpha,\bu)};\CF)
\DH^g_\BP(\dif \alpha). 
\end{eqnarray}
\end{ex}

\section{Optimal degenerations of Fano threefolds in family №2.23(a)} 
\label{Section: Optimal deg of 2.23(a)}
In this section, we will determine the optimal degenerations of Fano threefolds $X$ in the subfamily №2.23(a). By Theorem \ref{Theorem: H minimizer = central fiber K-ss}, it suffices to construct a special degeneration $(\CX, \xi_0)$ of $X$, such that the central fiber $(X_0,\xi_0)$ is weighted K-semistable. We achieve this (see Theorem \ref{Theorem: K-ps of optimal degeneration of №2.23(a)}) by a detail analysis of local weighted $\delta$-invariant via the weighted Abban-Zhuang estimate (see Section \ref{Section: compute S^g of №2.23(a)}). 
Our method actually works for possibly singular biconic curves $C\in |\CO(2,2)|$. We collect some facts on GIT stability of biconic curves in Section \ref{Subsection: singularities and GIT of biconic} and relate it to the weighted K-moduli space of $(X_0, \xi_0)$ in Section \ref{Subsection: Wmoduli2.23(a)}.

\subsection{The $\ord_{\tilde{H}}$-degeneration} \label{sec:H minimizer 2.23(a)}

Recall that any Fano threefold $X$ in family №2.23 is obtained by blowing up the quadric threefold $Q$ along a smooth elliptic curve $C$, which is a complete intersection of a hyperplane section $H\in |\CO_Q(1)|$ and a quadric section $Q'\in |\CO_Q(2)|$. The family №2.23 is divided into two subfamilies by smoothness of $H$, 
\begin{itemize}
    \item №2.23 (a), $H\cong \IP^1\times \IP^1$, 
    \item №2.23 (b), $H\cong \IP(1,1,2)$.
\end{itemize}

For any $X$ in №2.23(a), by choosing a suitable coordinate, we may assume that 
$$Q = \{u\cdot l_1+xy-zw = 0\}\seq \IP^4_{u,x,y,z,w}, \quad
H = \{u=0\}|_Q, \quad 
Q'= \{u\cdot l_2+ q = 0\}|_Q,  $$ 
where $l_i=l_i(u,x,y,z,w)$ are linear and $q=q(x,y,z,w)$ is quadric. Then $H = \{xy-zw = 0\} \subset \IP^3$ is a smooth quadric surface, and $C = Q'|_H = \{q=0\}|_H \seq H$ is a smooth curve. We identify $H$ with $\IP^1 \times \IP^1$ by the Segre embedding, 
$$\IP^1_{X,U} \times \IP^1_{Y,V} \to \IP^3_{x,y,z,w}, \quad
[X,U],[Y,V] \mapsto [XV, YU, UV, XY]. $$ 
Hence $q=q(XV, YU, UV, XY) \in H^0(\IP^1\times \IP^1, \CO(2,2))$, and $C\seq \IP^1\times \IP^1$ is a {\it biconic curve}. 

For any $X$ in №2.23(b), we assume by choosing coordinate that 
$$Q = \{u\cdot l_1 + xy-z^2 = 0\}\seq \IP^4_{u,x,y,z,w}, \quad
H = \{ u = 0 \}|_Q, \quad
Q' = \{u\cdot l_2 + q = 0\}|_Q, $$ 
where $l_1, l_2, q$ are the same as above. In this case $H=\{xy-z^2=0\}\seq \IP^3_{x,y,z,w}$ is a cone over the plane conic $\{xy-z^2=0\}\seq \IP^2_{y,z,w}$. We also identify $H$ with $\IP(1,1,2)$ via the embedding, 
$$\IP(1,1,2)_{X,Y,Z} \to \IP^3_{x,y,z,w}, \quad
[X,Y,Z]\mapsto [X^2, Y^2, XY, Z]. $$
Hence $q=q(Z, XY, X^2, Y^2) \in H^0(\IP(1,1,2), \CO(4))$, and and $C\seq \IP(1,1,2)$ is a quartic curve.

Consider the degeneration $\CX$ induced by the divisorial valuation $\ord_H$, which is given by the $\IG_m$-action $\lam(t)\cdot [u,x,y,z,w] = [t^{-1}u,x,y,z,w]$. Note that 
\begin{eqnarray*}
\lam(t)^*(u\cdot l_1+xy-zw)
&=&
tu \cdot \lam(t)^*l_1+xy-zw, \\
\lam(t)^*(u\cdot l_1+xy-z^2)
&=&
tu \cdot \lam(t)^*l_1+xy-z^2.
\end{eqnarray*}
Hence $Q_0 = \lim_{t\to 0}\lam(t)^* Q = \{xy-zw = 0\} \seq \IP^4_{u,x,y,z,w}$ for №2.23(a), and $Q_0=\{xy-z^2 = 0\} \seq \IP^4_{u,x,y,z,w}$ for №2.23(b). 
Also, note that $C$ lies in the fixed locus of the $\IG_m$-action. Hence $\lam(t)$ induces a special degeneration $(\CX,\xi)$ of $X = \Bl_C Q$ to $X_0:=\Bl_{C} Q_0$. Both $Q_0$ and $X_0$ admit the $\IG_m$-action. Let $\xi_0$ be the soliton candidate of $X_0$ with respect to the $\IG_m$-action. Then there exists $a_0>0$ such that $\xi_0=a_0 \xi$. 
\begin{thm}
For any Fano threefold $X$ in №2.23(a), the special degeneration $(\CX,\xi_0)$ is the optimal degeneration of $X$, where $\xi_0$ is the soliton candidate of $X_0=\CX_0$ with respect to the $\IG_m$-action. 
\end{thm}
In the following of this subsection, we focus on the case of №2.23(a). We first compute the soliton candidate $\xi_0$ of $X_0$ with respect to the $\IG_m$-action with the same strategy of \cite[Section 6]{MW23}. We will choose a suitable admissible flag, then construct an Okounkov body of $X_0$ compatible with the $\IG_m$-action. Working locally, on the affine chart $U_z=\{z\ne 0\} \seq \IP^4_{u,x,y,z,w}$ (just assume that $z=1$), we have $Q_0 = \IA^3_{u,x,y}$ with $w=xy$. Hence, over $U_z$ we have $X_0 = \Bl_{C} \IA^3_{u,x,y}\seq \IA^3\times \IP^1_{\zeta_0, \zeta_1}$ given by $X_0= \{\zeta_0q=\zeta_1 u\}$. We further restrict ourselves to the affine chart $\{\zeta_1 \ne 0\}$. Then $u = s q$, where $s=\zeta_0/\zeta_1$. We have $X_0 = \IA^3_{s,x,y}, q=q(x,y,1,xy)$, and 
$$\pi: X_0 = \IA^3_{s,x,y} \to \IA^3_{u,x,y} = Q_0, \quad 
  (s,x,y) \mapsto (qs, x, y). $$ 
Then $E = \{q = 0\}$ is the exceptional divisor and  $\tH = \{s = 0\}$ is the strict transform of $H$. Note that $-K_{X_0} = \pi^*(-K_{Q_0})-E = \pi^*\mathcal{O}_{Q_0}(3)-E$. Hence
\begin{eqnarray*}
H^0(X_0,-K_{X_0}) 
&=& \Big(u\cdot \pi^*H^0(Q_0,\CO_{Q_0}(2))\oplus q\cdot \pi^* H^0(Q_0,\CO_{Q_0}(1))\Big)\cdot q^{-1}\\
&=& s\cdot \pi^*H^0(Q_0,\CO_{Q_0}(2))  \oplus \pi^*H^0(Q_0,\CO_{Q_0}(1))\\
&=& \la s^3 q^2 \ra \oplus s^2 q \cdot \la 1,x,y,xy \ra \oplus s\cdot \la 1,x,y,xy \ra^2 \oplus \la 1,x,y,xy \ra\\
&=& (R_1)_3 \oplus (R_1)_2 \oplus (R_1)_1 \oplus (R_1)_0,
\end{eqnarray*}
where $(R_1)_{\alpha}=\{s\in R_1=H^0(X_0, -K_{X_0})\mid \forall\, t\in \IG_m, \,t\cdot s = t^{\alpha}s\}$ is the weight space. This is clearly a decomposition of $\IG_m$-invariant subspaces. 
We make a remark that the anti-canonical ring $R=R(X_0)$ is generated by $R_1=H^0(X_0,-K_{X_0})$. 

Without loss of generality, we assume that $q= ax+by+ \text{higher terms}$, with $a,b\ne0$. Consider the admissible flag $\IA^3_{s,x,y}\supseteq \{s=0\}\supseteq \{s=x=0\}\supseteq \{s=x=y=0\}$, which induces a faithful valuation $\fv: s\mapsto (1,0,0), x\mapsto (0,1,0), y\mapsto (0,0,1)$. It is clear that the faithful valuation $\nv$ is adapted to the $\IG_m$-action from the construction. By assumption we have $\fv(q) = (0,0,1)$. 
Therefore $\fv(R_1)$ is given by
\begin{eqnarray*}
    s^3\cdot q^2 &\mapsto & (3,0,2);\\
    s^2\cdot q\cdot \la 1,x,y,xy\ra &\mapsto & (2,0,1), (2,1,1), (2,0,2), (2,1,2); \\
    s\cdot \la 1,x,y,xy \ra^2 &\mapsto & (1,0,0), (1,1,0),(1,0,1),(1,1,1),(1,2,0),(1,2,1), \\
    &&(1,0,2),(1,1,2),(1,2,2);\\
    \la 1,x,y,xy \ra &\mapsto & (0,0,0), (0,1,0), (0,0,1), (0,1,1).
\end{eqnarray*}
The convex hull of these points is an Okounkov body of $X_0$ since $R$ is generated by $R_1$. 
In order to make the Okounkov body compatible with the $\IG_m$-action, we may shift the first coordinate by $A_{X_0}(\tH)=1$ (see for example \cite[Section 2.4]{Wang24}), and get a polytope $\BO$ which lives over $s\in [-1, 2]$ (the moment polytope of the $\IG_m$-action). 
This Okounkov body $\BO$ is spanned by vertices 
$$(-1,0,0),(-1,0,1),(-1,1,0),(-1,1,1), 
(0,0,0),(0,0,2),(0,2,0),(0,2,2), 
(2,0,2). $$
Hence the $\IG_m$-equivariant Duistermaat-Heckman measure (which corresponds to the slice volumes of $\BO$)  on the moment polytope $\BP=[-1,2]$ is 
\begin{eqnarray*}
\DH_\BP(\dif s) = 
\left\{ \begin{array}{ll}
(2+s)^2 & -1\le s\le 0,\\
(2-s)^2 &  0\le s\le 2.
\end{array} \right. 
\end{eqnarray*} 
Note that here the $\DH$-measure in unnormalized and satisfies $\int_{\BP} \DH_{\BP}(\dif s)=\vol_{\IR^n}(\BO)$. Recall that $\BH(\wt_{\xi})=\log\big(\int_\BP e^{-\xi s} \cdot \DH_\BP(\dif s)\big)$ is a strictly convex functional on $N_{\IR}$, so it admits a unique minimizer $\xi_0$, which we call it the soliton candidate in Section \ref{sec:weighted}. Furthermore, the minimizer should satisfy the condition that the corresponding Tian-Zhu invariant vanishes, namely, $\nabla_{\xi} H(\xi_0)=-\TZ_{g_{\xi_0}}(\xi)=0$, which implies $\int_{-1}^0 s\cdot e^{-\xi_0 s}(2+s)^2\, \dif s +\int_0^2 s\cdot e^{-\xi_0 s}(2-s)^2 \, \dif s=0$. We can numerically solve the equation and obtain
\begin{eqnarray}
\label{Number: soliton candidate of №2.23(a)}
\xi_0 \approx  0.2737918510108124.
\end{eqnarray} 
Let $g=g_{\xi_0}: \BP \rightarrow \IR_{>0}$ be the weight function associated to the soliton candidate $\xi_0 \in N_{\IR}$. By Formula (\ref{eq:S^g}), we have $S^g(R_{\bu}; \tilde{H})=\frac{1}{\Bv^g}\int_{\BP} \alpha\cdot e^{-\langle \xi_0, \alpha \rangle} \DH_{\BP}(\dif \alpha) =1 = A_{X_0}(\tH)$.
\begin{thm}
\label{Theorem: K-ps of optimal degeneration of №2.23(a)}
The pair $(X_0, \xi_0)$ is weighted K-polystable. 
\end{thm}
\begin{proof}
The argument is similar to that of \cite[Theorem 6.1]{MW23}. We sketch the proof for the reader's convenience. Let $W_\bu$ be the refinement of $R_\bu = R(X_0, -K_{X_0})$ by $\tH\cong \IP^1\times \IP^1$. By the choice of $g$ (or the soliton candidate) we see that $S^g(R_\bu; \tH)= A_{X_0}(\tH)$. Note that $X_0$ admits only one-dimensional torus action. By the weighted Abban-Zhuang estimate \cite[Theorem 4.6]{MW23}, we have 
$$\delta^g_{p,\IG_m}(X_0;R_\bu) \ge \min\{1, \delta^g_p(\IP^1\times \IP^1; W_\bu)\}. $$
By Remark \ref{Remark: 2.23(a) (1,1)-blowup} and \ref{Remark: 2.23(a) (2,1)-blowup} in Section \ref{Section: compute S^g of №2.23(a)}, we have $\delta^g(\IP^1\times \IP^1; W_\bu) \approx \frac{3}{2.782066} >1$. Then $(X_0,\xi_0)$ is weighted K-polystable with the same argument in \cite[Theorem 6.1]{MW23}. Indeed, if $(X_0,\xi_0)$ is strictly weighted K-semistable, then by \cite[Lemma 4.14]{BLXZ23}, there exists a $\IG_m$-invariant quasi-monomial valuation $v$ over $X_0$ such that $S^g(R_\bu; v)= A_{X_0}(v)$. Since every point of $\tH$ is invariant under the $\IG_m$-action, we see that $\Val^{\IT,\circ}_{X_0} \cong \Val^\circ_{\tH}\times N_\IR(\IG_m)$. Hence there exists $v_0\in \Val^\circ_{\tH}$ whose pull-back to $X_0$ is just $v$. Thus
$$1=\frac{A_{X_0}(v)}{S^g(R_\bu; v)} = \frac{A_{\tH}(v_0)}{S^g(W_\bu; v_0)} \ge \delta^g(\IP^1\times \IP^1; W_\bu) >1, $$
which is a contradiction. 
\end{proof}

\subsection{Singularities and GIT stability of biconic curves in $\IP^1\times\IP^1$}
\label{Subsection: singularities and GIT of biconic}
In the work of \cite{CDF+}, the singularities and GIT stability of biconic curves $C\in |\CO(2,2)|$ in $\IP^1 \times \IP^1$ have been classified. For the convenience of readers, we collect these results and fix some notation. Let $p\in C$ be a point. We may choose a coordinate such that $p=([0,1],[0,1])$. We work on the affine chart $\{U\ne 0, V\ne0\}\cong \IA^2_{X,Y}$ (just assume that $U=V=1$). Then $C=\{f=0\}$, where $f=\sum_{0\le i,j\le2} a_{ij}X^iY^j$ and $a_{00}=0$. 
Assume that $p\in C$ is singular in the following, which means that $a_{10}=a_{01}=0$. 

{\bf Case (I). }
We first assume that $C$ contains no ruling component passing through $p$, which is equivalent that $a_{20}$ and $a_{02}$ are both non-zero. In this case, the degree two part of $f$ is 
\begin{eqnarray*}
a_{20}X^2 + a_{11}XY + a_{02}Y^2
= a_{20}(X+b_1Y)(X+b_2Y), 
\end{eqnarray*}
where $b_1,b_2\ne 0$. Then there are three possibilities, 
\begin{itemize}
\item {\bf (Ia).} $b_1\ne b_2$. Then $p\in C$ is nodal; 
\item {\bf (Ib).} $b_1=b_2$ and $C$ is reduced. Then $p\in C$ is cuspidal; 
\item {\bf (Ic).} $b_1=b_2$ and $C$ is non-reduced. Hence $f=a_{20}(X+b_1Y)^2$. 
\end{itemize}

{\bf Case (II). } If $C$ contains a ruling component passing through $p$, we just assume that $a_{20}=0$. Then
\begin{eqnarray*}
f= X\cdot(a_{22}XY^2+ a_{21}XY+ a_{12}Y^2+ a_{20}X + a_{11}Y). 
\end{eqnarray*}
In this case, $l_1=\{X=0\}$ is a component of $C$. We see that $C=l_1+C'$ for some $C'\in |\CO(1,2)|$. 
There are four possibilities, 
\begin{itemize}
\item {\bf (IIa)}. $a_{11}\ne 0$. Then $l_1$ and $C'$ intersect transversally at $p\in C$. Hence $p\in C$ is nodal; 
\item {\bf (IIb)}. $a_{11}=0$ and $a_{12},a_{20}\ne 0$. Then $l_1$ is tangent to $C'$ at $p\in C$ of multiplicity two. Hence $p\in C$ is cuspidal; 
\item {\bf (IIc)}. $a_{11}=a_{20}=0$ and $a_{21}, a_{12}\ne0$. Then $f= XY(a_{22}XY+ a_{21}X+ a_{12}Y)$ and $p\in C$ is a triple point; 
\item {\bf (IId)}. $a_{11}=a_{12}=0$. Then $f= X^2(a_{22}Y^2+ a_{21}Y+ a_{20})$ and $p\in C$ is contained in a multiple ruling component (hence is at least a triple point). 
\end{itemize}

Next we recall the GIT-stability of $C\in|\CO(2,2)|$ is defined under the $\SL_2\times\SL_2$ action. For any $a,b\in \IZ$, we consider the one-parameter subgroup defined by 
$$
\lam(t)\cdot([X,U],[Y,V])=([t^{-a}X, t^aU],[t^{-b}Y, t^{-a}V]).
$$ 
Then for any biconic $f=\sum_{0\le i,j\le2} a_{ij}X^iY^jU^{2-i}V^{2-j}$ we have
\begin{eqnarray*}
\lam(t)^*f = 
\sum_{0\le i,j\le2} 
t^{(2i-2)a+(2j-2)b}a_{ij}X^iY^jU^{2-i}V^{2-j}. 
\end{eqnarray*}
Writing in matrices, we have 
\begin{eqnarray*}
\lam(t)\cdot
\left( \begin{array}{ccc}
a_{22} & a_{21} & a_{20} \\
a_{12} & a_{11} & a_{10}  \\
a_{02} & a_{01} & a_{00}   
\end{array} \right)
=
\left( \begin{array}{ccccc}
t^{2a+2b}a_{22} & t^{2b}a_{21} & t^{-2a+2b}a_{20} \\
t^{2a}a_{12} & a_{11} & t^{-2a}a_{10}  \\
t^{2a-2b}a_{02} & t^{-2b}a_{01} & t^{-2a-2b}a_{00}   
\end{array} \right). 
\end{eqnarray*}

\begin{thm}
\label{Theorem: GIT of biconic curves}
For any biconic curve $C\in|\CO(2,2)|$, we have

(a). it is GIT-stable if and only if it is smooth; 

(b). it is strictly GIT-semistable if and only if
\begin{itemize}
\item $C=l+C'$ where $l\in |\CO(1,0)|$ or $|\CO(0,1)|$, $C'$ is smooth and they intersect transversally, or
\item $C=l_1+l_2+C'$ where $l_1\in |\CO(1,0)|, l_2 \in |\CO(0,1)|$, and $l_1\cap l_2\cap C'=\varnothing$; 
\end{itemize}

(c). it is GIT-polystable if and only if 
\begin{itemize}
\item $C=l_1+l_1'+l_2+l_2'$ for different rulings $l_1,l_1'\in|\CO(1,0)|, l_2,l_2'\in|\CO(0,1)|$, or
\item $C=C_1+C_2$ for some smooth $C_1,C_2\in |\CO(1,1)|$ (may not be different). 
\end{itemize}
\end{thm}

\begin{rmk}\rm 
The singularities in {\bf Case (I)} are GIT-polystable, {\bf Case (IIa)} are strictly GIT-semistable, and {\bf Case (IIb)}, {\bf (IIc)} and {\bf (IId)} are GIT-unstable. The relationship between GIT-stability and singularity types is summarized as follows.

\begin{table}[htbp]
    \centering
    \renewcommand{\arraystretch}{1.2}
    \caption{Singularities and GIT Stability of Biconic Curves in $\mathbb{P}^1 \times \mathbb{P}^1$}
    \label{tab:biconic_git_stability}
    \begin{tabular}{clll}
        \hline
        \textbf{Case} & \textbf{Singularity Type at $p$} & \textbf{Geometric Condition} & \textbf{GIT Stability} \\
        \hline
        -- & Smooth & Smooth curve & GIT-stable \\
        (Ia) & Nodal & No ruling component through $p$ & GIT-polystable \\
        (Ib) & Cuspidal & No ruling component through $p$ & GIT-polystable \\
        (Ic) & Non-reduced & No ruling component through $p$ & GIT-polystable \\
        (IIa) & Nodal & Ruling component (transversal) & Strictly GIT-semistable \\
        (IIb) & Cuspidal & Ruling component (tangent) & GIT-unstable \\
        (IIc) & Triple point & Ruling component & GIT-unstable \\
        (IId) & $\ge$ Triple point & Multiple ruling component & GIT-unstable \\
        \hline
    \end{tabular}
\end{table}
\end{rmk}

\subsection{Computing $S^g$ for the degeneration $X_0$ of №2.23(a)}
\label{Section: compute S^g of №2.23(a)}
We work in a little more general setting, that is, $C\in |\CO(2,2)|$ (may not be smooth).  
Let $W_\bu$ be the refinement of $R_\bu=R(X_0, -K_{X_0})$ by the toric divisor $\tH = \{s = 0\} \cong \IP^1 \times \IP^1$, which is the strict transform of the hyperplane section $\{u=0\}$ of $Q_0$. Then 
\begin{eqnarray*}
W_{(1,s)} = 
\left\{ \begin{array}{ll}
H^0\Big(\IP^1\times \IP^1, \CO(2+s,2+s)\Big) 
& -1\le s < 0,\\
H^0\Big(\IP^1\times \IP^1, \CO(2-s,2-s)\Big) + s\cdot C
& 0\le s\le 2. 
\end{array} \right. 
\end{eqnarray*} 
where $C = H \cap E_C$. Let $p\in \IP^1\times \IP^1$ and $l_1\in |\CO(1,0)|, l_2\in |\CO(0,1)|$ be two rulings containing $p$.

{\bf Case (A). Refining by a ruling.} We first take refinement of $W_\bu$ by $l=l_1$. The filtration is 
\begin{eqnarray*}
\CF^{(t)}_l W_{(1,s)} = 
\left\{ \begin{array}{ll}
H^0\Big(\IP^1\times \IP^1, \CO(2+s-t,2+s)\Big) + t l
& -1\le s < 0, 0\le t\le 2+s,\\
H^0\Big(\IP^1\times \IP^1, \CO(2-s-t, 2-s)\Big) + t l+s C
& 0\le s\le 2, 0\le t\le 2-s. 
\end{array} \right. 
\end{eqnarray*} 

By Example \ref{prop:SgW}, we have 
\begin{eqnarray*} 
S^g(W_\bu; l) 
&=&\frac{1}{2!} 
\frac{1}{\Bv^g} 
\int_\IR
\int^2_{-1} g(s)
\vol(\CF^{(t)}W_{(1,s)}) 
 \dif s \dif t \\
&=& \frac{1}{2\Bv^g} 
\Big(
\int_{-1}^0 g(s) \int_0^{2+s} 2(2+s-t)(2+s) \dif t  \dif s \\
&&+ \int_0^2 g(s) \int_0^{2-s} 2(2-s-t)(2-s) \dif t  \dif s 
\Big) \\
&\approx& 0.782066 \,\, <1 = A_H(l), 
\end{eqnarray*}
where $g(s) = e^{-\xi_0 s}$ and $\Bv^g \approx 4.94383$.
Moreover, let $p_1,p_2\in C$ be the points that $l$ and $C$ intersect at ($p_1=p_2$ if $l$ is tangent to $C$). Then
\begin{eqnarray*}
W^l_{(1,s,t)} = 
\left\{ \begin{array}{ll}
H^0\Big(\IP^1, \CO(2+s)\Big) 
& -1\le s < 0, 0\le t\le 2+s,\\
H^0\Big(\IP^1, \CO(2-s)\Big) + s(p_1+p_2)
& 0\le s\le 2, 0\le t\le 2-s. 
\end{array} \right. 
\end{eqnarray*} 

Next we take refinement by $p\notin C$. We have 
\begin{eqnarray*}
\CF^{(x)}_p W^l_{(1,s,t)} = 
\left\{\begin{array}{ll}
H^0\Big(\IP^1, \CO(2+s-x)\Big) + x p 
& (s,t,x)\in \D_-,\\
H^0\Big(\IP^1, \CO(2-s-x)\Big) + x p+s(p_1+p_2)
& (s,t,x)\in \D_+. 
\end{array} \right. 
\end{eqnarray*} 
where
\begin{eqnarray*}
\D_-
&=&
\{-1\le s<0, 0\le t\le 2+s, 0\le x\le 2+s\},  \\
\D_+ 
&=&
\{0\le s\le 2, 0\le t\le 2-s, 0\le x\le 2-s\}. 
\end{eqnarray*} 
Hence 
\begin{eqnarray*} 
S^g(W^l_\bu; p) 
&=&\frac{1}{1!} 
\frac{1}{\Bv^g} 
\int^3_{-1} g(s)
\Big(
\int_{\D_-\cup \D_+}
\vol(\CF^{(x)}W_{(1,s,t)}) 
\dif x \dif t \Big)
 \dif s  \\
&=& \frac{1}{\Bv^g} 
\Big(
\int_{-1}^0 g(s) 
\int_0^{2+s}
\int_0^{2+s}
 (2+s-x) \dif x  \dif t  \dif s  \\
&&\quad 
+\int_0^2 g(s) 
\int_0^{2-s} 
\int_0^{2-s}
(2-s-x) \dif x  \dif t  \dif s 
\Big) \\
&\approx& 0.782066 \,\, <1 = A_l(p). 
\end{eqnarray*}

\begin{rmk}\rm
We conclude in {\bf Case (A)} that 
$\delta^g_p(\IP^1\times\IP^1; W_\bu) > 1 $ for any $C\in|\CO(2,2)|$ and $p\notin C$. 
\end{rmk}

{\bf Case (B). Refining by the exceptional line of $(1,1)$-blowup.} Assume that the point $p \in C$ is of multiplicity $1\le k \le 4$. Let $\nu: \Bl_p(\IP^1 \times \IP^1) \rightarrow \IP^1 \times \IP^1$ be the ordinary blowup of $p$ and we denote the exceptional line by $E \cong \IP^1$. Let $\tilde{l}_1, \tilde{l}_2, \tilde{C}$ be the strict transform of $l_1, l_2, C$ on the surface $\Bl_p(\IP^1 \times \IP^1)$ respectively. We have the pull-backs $l_i = \tilde{l}_i + E\, (i=1,2)$ and $C = \tilde{C}+ k E$. 
The negative and positive parts in the Zariski decomposition of $\CF_E^{(t)}W_{(1,s)}$ are
\begin{eqnarray*}
N(\CF_E^{(t)}W_{(1,s)})  = 
\left\{\begin{array}{ll}
t E
& -1\leq s \leq 0, 0\leq t \leq 2+s,\\
t E+(t-2-s)(\tilde{l}_1 + \tilde{l}_2) 
& -1\leq s \leq 0, 2+s\leq t \leq 4+2s,\\
k s E + s \tilde{C} 
& 0\leq s \leq 2, 0\leq t \leq k s,\\
t E + s \tilde{C} 
& 0\leq s \leq 2, k s\leq t \leq 2+(k-1)s,\\
t E+ s\tilde{C}
&\\
+(t-2-(k-1)s)\cdot (\tilde{l}_1 + \tilde{l}_2) 
& 0\leq s \leq 2, 2+(k-1)s\leq t \leq 4+(k-2)s, 
\end{array} \right.
\end{eqnarray*}

\begin{eqnarray*}
P(\CF_E^{(t)}W_{(1,s)})  = 
\left\{\begin{array}{ll}
\CO(2+s) - tE, \\
\CO(2+s) - tE - (t-2-s)(\tilde{l}_1+\tilde{l}_2), \\
\CO(2-s), \\
\CO(2-s) - (t - ks)E, \\
\CO(2-s) - (t - ks)E - (t - 2 - (k-1)s)(\tilde{l}_1+\tilde{l}_2). 
\end{array} \right. 
\end{eqnarray*}
And the volume follows 
\begin{eqnarray*}
\vol(\CF_E^{(t)}W_{(1,s)})  = 
\left\{\begin{array}{ll}
2(2+s)^2 - t^2, \\
2(2+s)^2 - t^2 + 2(t-2-s)^2, \\
2(2-s)^2, \\
2(2-s)^2 - (t - ks)^2, \\
2(2-s)^2 - (t - ks)^2 + 2(t - 2 - (k-1)s)^2. 
\end{array} \right. 
\end{eqnarray*}
Therefore 
\begin{eqnarray*} 
S^g(W_\bu; E) 
&=&\frac{1}{2!} 
\frac{1}{\Bv^g} 
\int_\IR
\int^2_{-1} g(s)
\vol(\CF^{(t)}W_{(1,s)}) 
 \dif s \dif t \\
&=& k + (2-k)q, 
\end{eqnarray*}
where $q\approx 0.782066$ and $1\le k\le 4$. Hence 
$$\frac{A_{\IP^1\times \IP^1}(E)}{S^g(W_\bu; E)} >, =, < 1 \text{ if } k=1,k=2, k\ge 3 \text{ respecitively}. $$
In particular, if $C$ has triple or quadruple points, then $(X_0,\xi_0)$ is weighted K-unstable. 

Next, we take refinement of $W_\bu$ by $E$. Assume that $l_1, l_2$ are not tangent to any germ of $C$ at $p$. 

If $k=1$, we denote by $p_1 = \tilde{l}_1 \cap E, p_2 = \tilde{l}_2 \cap E, p_3 = \tilde{C}\cap E$, which are different by assumption. Then
\begin{eqnarray*}
W^E_{(1,s,t)}  = 
\left\{\begin{array}{ll}
H^0(\IP^1,\CO(t))
& -1\leq s \leq 0, 0\leq t \leq 2+s,\\
H^0(\IP^1, \CO(4+2s-t)) + (t-2-s)(p_1+p_2)
& -1\leq s \leq 0, 2+s\leq t \leq 4+2s,\\
H^0(\IP^1, \CO(t-s)) + s p_3 
& 0\leq s \leq 2, s\leq t \leq 2,\\
H^0(\IP^1, \CO(4-s-t)) + s p_3+ (t-2)(p_1+p_2)
& 0\leq s \leq 2, 2\leq t \leq 4-s.
\end{array} \right.  
\end{eqnarray*}
Hence we have 
$$
S^g(W^E_{\bu};p)\approx 0.521377, \,\,
S^g(W^E_{\bu}; p_3) \approx 0.739311, \,\,
S^g(W^E_{\bu}; p_1) 
=S^g(W^E_{\bu}; p_2) \approx 0.782066 \,\,$$
for any $p \in E\setminus \{p_1, p_2, p_3\}$.  

If $k=2$, recall that $C_1, C_2$ are the two germs of $C$ at $p$, and let $p_3=\tilde{C}_1\cap E, p_4=\tilde{C}_2\cap E$. The series $W^E_{(1,s,t)}$ is the same as above for $-1\le s\le0$. We only present the series for $0\le s\le 2$.  
\begin{eqnarray*}
W^E_{(1,s,t)}  = 
\left\{\begin{array}{ll}
H^0(\IP^1, \CO(t-2s)) + s(p_3+p_4) 
& 2s\leq t \leq 2+s,\\
H^0(\IP^1, \CO(4-t)) + s(p_3+p_4)+ (t-2-s)(p_1+p_2)
& 2+s\leq t \leq 4.
\end{array} \right.  
\end{eqnarray*}
Hence if $p_3\ne p_4$, we have 
$$S^g(W^E_\bu; p_1)=S^g(W^E_\bu; p_2)\approx 0.782066, \,\,
 S^g(W^E_\bu; p_3)=S^g(W^E_\bu; p_4)\approx 0.739311. $$ 
Otherwise $p_3=p_4$, that is, $p\in C$ is a singular point as {\bf Case (Ib)} and {\bf Case (Ic)} in Section \ref{Subsection: singularities and GIT of biconic}, then 
$$S^g(W^E_\bu; p_3)\approx 0.957246.$$

\begin{rmk}\rm
\label{Remark: 2.23(a) (1,1)-blowup}
We conclude in {\bf Case (B)} that
\begin{itemize}
\item if $p\in C$ is smooth, and $l_1,l_2$ are not tangent to any germ of $C$ at $p$, then $$\delta^g_p(\IP^1\times\IP^1; W_\bu) \approx \frac{1}{0.782066} > 1; $$ 
\item if $p\in C$ is of multiplicity two, and $l_1,l_2$ are not tangent to any germ of $C$ at $p$, then $$\delta^g_p(\IP^1\times\IP^1; W_\bu) 
=\min\{1, \frac{1}{0.957246}\}= 1; $$ 
\item  if $p\in C$ is of multiplicity at least three, then $(X_0,\xi_0)$ is weighted K-unstable. 
\end{itemize}
\end{rmk}

{\bf Case (C). Refining by the exceptional line of $(2,1)$-blowup.} We consider the case that the ruling $l_1$ is tangent to one germ of $C$ at $p$. Let $k=\mult_p(l_1\cap C)$. Then $2\le k\le 4$. 
Let $E$ be the exceptional line of the $(2,1)$-blowup of $p$ such that $\ord_E(l_1)=2, \ord_E(l_2)=1$. We have $l_1 = \tilde{l}_1 + 2E, l_2 = \tilde{l}_2 + E, C=\tilde{C}+kE$. There are three cases
\begin{itemize}
\item if $k=2$, then $C=C_1$ is smooth at $p$;
\item if $k=3$, then $p\in C$ is nodal, $l_1$ is tangent to one germ $C_1$ of $C$ at $p$, and intersects the other germ $C_2$ transversally at $p$;
\item if $k=4$, then $p\in C$ is cuspidal, or $2l_1$ is a component of $C$, see {\bf Case (IIb)} and {\bf (IId)} in Section \ref{Subsection: singularities and GIT of biconic}. In this case, $l_1$ is tangent to both germs of $C$ at $p$. 
\end{itemize}

The intersection numbers of curves $\tilde{l}_1, \tilde{l}_2, E$ can be computed using weighted Segre class (see for example \cite[Lemma 3.19]{LX18}). We collect the results in the following tables
\begin{table}[!htp]
\begin{floatrow}
\begin{tabular}{|c|c|c|c|}
\hline
      & $l_1$ & $l_2$ & $E$  \\ \hline
$l_1$ & $0$     & $1$     & $0$  \\ \hline
$l_2$ & $1$     & $0$     & $0$  \\ \hline
$E$   & $0$     & $0$     & $-\frac{1}{2}$ \\ \hline
\end{tabular} \quad
\begin{tabular}{|c|c|c|c|}
\hline
              & $\tilde{l}_1$ & $\tilde{l}_2$   & $E$              \\ \hline
$\tilde{l}_1$ & -2            & 0               & 1              \\ \hline
$\tilde{l}_2$ & 0             & $-\frac{1}{2}$  & $\frac{1}{2}$  \\ \hline
$E$             & 1             & $\frac{1}{2}$   & $-\frac{1}{2}$ \\ \hline
\end{tabular} \,\,. 
\end{floatrow}
\end{table}

The negative part of $\CF_E^{(t)}W_{(1,s)}$ is as follows. For $-1\le s\le 0$,
\begin{eqnarray*}
N(\CF_E^{(t)}W_{(1,s)})  = 
\left\{\begin{array}{ll}
t E
& 0\leq t \leq 2+s,\\
t E+ \frac{1}{2}(t-2-s)\tilde{l}_1
& 2+s\leq t \leq 4+2s,\\
t E+ \frac{1}{2}(t-2-s)\tilde{l}_1 + (t-4-2s)\tilde{l}_2 
& 4+2s\leq t \leq 6+3s, 
\end{array} \right.
\end{eqnarray*}
and for $0\le s\le 2$,
\begin{eqnarray*}
N(\CF_E^{(t)}W_{(1,s)})  = 
\left\{\begin{array}{ll}
s\tilde{C} + ks E
& 0\leq t \leq ks,\\
s\tilde{C} + t E
& ks \leq t \leq 2+(k-1)s,\\
s\tilde{C} + t E + \frac{1}{2}(t-2-(k-1)s)\tilde{l}_1
& 2+(k-1)s\leq t \leq 4+(k-2)s,\\
s\tilde{C} + t E + \frac{1}{2}(t-2-(k-1)s)\tilde{l}_1 &\\
+ (t-4-(k-2)s)\tilde{l}_2 
& 4+(k-2)s\leq t \leq 6+(k-3)s. 
\end{array} \right.
\end{eqnarray*}
Then we have the positive part, for $-1\le s\le 0$,
\begin{eqnarray*}
P(\CF_E^{(t)}W_{(1,s)})  = 
\left\{\begin{array}{ll}
\CO(2+s) - tE, \\
\CO(2+s) - tE - \frac{1}{2}(t-2-s)\tilde{l}_1,\\
\CO(2+s) - tE - \frac{1}{2}(t-2-s)\tilde{l}_1 - (t-4-2s)\tilde{l}_2, 
\end{array} \right.
\end{eqnarray*}
and for $0\le s\le 2$,
\begin{eqnarray*}
P(\CF_E^{(t)}W_{(1,s)})  = 
\left\{\begin{array}{ll}
\CO(2-s), \\
\CO(2-s) - (t-ks) E,\\
\CO(2-s) - (t-ks) E - \frac{1}{2}(t-2-(k-1)s)\tilde{l}_1,\\
\CO(2-s) - (t-ks) E - \frac{1}{2}(t-2-(k-1)s)\tilde{l}_1 - (t-4-(k-2)s)\tilde{l}_2. 
\end{array} \right.
\end{eqnarray*}
Hence we have the volume, for $-1\le s\le 0$
\begin{eqnarray*}
\vol(\CF_E^{(t)}W_{(1,s)})  = 
\left\{\begin{array}{ll}
2(2+s)^2 - \frac{1}{2}t^2, \\
2(2+s)^2 - \frac{1}{2}t^2 + \frac{1}{2}(t-2-s)^2, \\
2(2+s)^2 - \frac{1}{2}t^2 + \frac{1}{2}(t-2-s)^2 + \frac{1}{2}(t-4-2s)^2,
\end{array} \right. 
\end{eqnarray*}
and for $0\le s\le 2$
\begin{eqnarray*}
\vol(\CF_E^{(t)}W_{(1,s)})  = 
\left\{\begin{array}{ll}
2(2-s)^2, \\
2(2-s)^2 - \frac{1}{2}(t - ks)^2, \\
2(2-s)^2 - \frac{1}{2}(t - ks)^2 + \frac{1}{2}(t - 2 - (k-1)s)^2, \\
2(2-s)^2 - \frac{1}{2}(t - ks)^2 + \frac{1}{2}(t - 2 - (k-1)s)^2 + \frac{1}{2}(t-4-(k-2)s)^2. 
\end{array} \right. 
\end{eqnarray*}
Therefore
\begin{eqnarray*} 
S^g(W_\bu; E) 
&=&\frac{1}{2!} 
\frac{1}{\Bv^g} 
\int_\IR
\int^2_{-1} g(s)
\vol(\CF^{(t)}W_{(1,s)}) 
 \dif s \dif t \\
&=& k+(3-k)q, 
\end{eqnarray*}
where $q\approx 0.782066$ and $k=2,3,4$. 
We see that 
$$\frac{A_{\IP^1\times \IP^1}(E)}{S^g(W_\bu; E)} >, =, < 1 \text{ if } k=2,3,4 \text{ respecitively}. $$
Hence $(X_0,\xi_0)$ is weighted K-unstable if $C$ has singularities in {\bf Case (IIb)} and {\bf (IId)} in Section \ref{Subsection: singularities and GIT of biconic}.

Next, we take refinement of $W_\bu$ by $E$ when $k=2$ or $3$. 

We denote by $p_1 = \tilde{l}_1 \cap E, p_2 = \tilde{l}_2 \cap E, p_3= \tilde{C_1}\cap E, p_4 = \tilde{C_2}\cap E$. Note that it is possible that $p_2=p_4$ if $l_2$ is tangent to $C_2$ at $p$. We denote the unique singular point of the $(2,1)$-blowup $\widetilde{\IP^1\times \IP^1}$ on $E$ by $p_0$. Hence for $-1\le s\le 0$, 
\begin{eqnarray*}
W^E_{(1,s,t)}  = 
\left\{\begin{array}{ll}
H^0(\IP^1,\CO(\frac{1}{2}t))
& 0\leq t \leq 2+s,\\
H^0(\IP^1, \CO(\frac{1}{2}(2+s))) + \frac{1}{2}(t-2-s)p_1
& 2+s\leq t \leq 4+2s,\\
H^0(\IP^1, \CO(\frac{1}{2}(6+3s-t))) &\\
+ \frac{1}{2}(t-2-s)p_1+ \frac{1}{2}(t-4-2s)p_2 
& 4+2s\leq t \leq 6+3s. 
\end{array} \right. 
\end{eqnarray*}
For $0\le s\le2$, we have
\begin{eqnarray*}
W^E_{(1,s,t)}  = 
\left\{\begin{array}{ll}
H^0(\IP^1, \CO(\frac{1}{2}(t-ks))) + s p_3 + \frac{1}{2}(k-2)sp_4  
& ks\leq t \leq 2+(k-1)s,\\
H^0(\IP^1, \CO(\frac{1}{2}(2-s))) + s p_3 + \frac{1}{2}(k-2)sp_4 &\\
+ \frac{1}{2}(t-2-(k-1)s)p_1
& 2+(k-1)s\leq t \leq 4+(k-2)s,\\
H^0(\IP^1, \CO(\frac{1}{2}(6+(k-3)s-t))) + s p_3 + \frac{1}{2}(k-2)sp_4 &\\
+ \frac{1}{2}(t-2-(k-1)s) p_1 +\frac{1}{2} (t-4-(k-2)s) p_2 
& 4+(k-2)s\leq t \leq 6+(k-3)s.
\end{array} \right.  
\end{eqnarray*}
Hence we have
$$S^g(W_{\bu}^E; p_0) \approx 0.325861, \,\,
S^g(W_{\bu}^E; p_1) \approx 0.782066, \,\,
S^g(W_{\bu}^E; p_2) \approx 0.391033, \,\,
S^g(W_{\bu}^E; p_3) \approx  0.543795. $$
If $k=3$, we have furthermore $S^g(W_{\bu}^E; p_4) \approx 0.434828$ when $p_4\ne p_2$, $S^g(W_{\bu}^E; p_4) = 0.5$ when $p_4=p_2$.  Since $\Diff_E(0) = \frac{1}{2}p_0$ and $p_0\notin \{p_1,p_2,p_3,p_4\}$, we have
\begin{eqnarray*} 
\delta^g(\IP^1, \frac{1}{2}p_0; W^E_{\bu}) 
=
\min\Big\{\frac{1-(1/2)}{S^g(W^E_\bu;p_0)}, 
\frac{1}{S^g(W^E_\bu;p_1)}\Big\} 
\approx \frac{1}{0.782066} >1. 
\end{eqnarray*} 
\begin{rmk}\rm
\label{Remark: 2.23(a) (2,1)-blowup}
We conclude in {\bf Case (C)} that 
\begin{itemize}
\item if $p\in C$ is smooth, and there exists a ruling through $p$ tangent to $C$ at $p$, then 
$$\delta^g_p(\IP^1\times\IP^1; W_\bu) \approx \frac{3}{2.782066} > 1 ; $$
\item if $p\in C$ is nodal, and there exists a ruling through $p$ tangent to some germ of $C$ at $p$, then 
$$\delta^g_p(\IP^1\times\IP^1; W_\bu) 
= \min\{1,\frac{1}{0.782066}\} =1; $$
\item  if $C$ is the union of a ruling $l$ and a curve tangent to $l$, or $C$ contains a double ruling $2l$ ({\bf Case (IIb)} and {\bf (IId)} in Section \ref{Subsection: singularities and GIT of biconic}), then $(X_0,\xi_0)$ is weighted K-unstable. 
\end{itemize}
\end{rmk}

\subsection{The weighted K-moduli space of $(X_0,\xi_0)$}\label{Subsection: Wmoduli2.23(a)}
By Theorem \ref{Theorem: K-ps of optimal degeneration of №2.23(a)} we see that $(X_0,\xi_0)$ is weighted K-polystable when $C$ is smooth. In this subsection, we discuss the case that $C\seq H$ is singular. We mainly focus on the cases that $A_{X_0}(E)/S^g(W_\bu;E)=1$ in the previous subsection. As a consequence, we prove that the weighted K-moduli space of $(X_0,\xi_0)$ is isomorphic to the GIT-moduli space of biconic curves $C\seq \IP^1\times\IP^1$. 

\begin{thm}
\label{Theorem: moduli of optimal degeneration of 2.23(a)}
Let $Q_0=C(\IP^1\times\IP^1, \CO(1,1))$ be a cone and $H\cong \IP^1\times \IP^1$ be the section at infinity. Let $X_0$ be the blowup of $Q_0$ along $C=Q'_0|_{H}$ where $Q'_0\in|\CO_{Q_0}(2)|$. Denote by $\xi_0$ the soliton candidate of $X_0$. 
Then $(X_0, \xi_0)$ is weighted K-semistable (weighted K-polystable) if and only if $C \seq H$ is GIT-semistable (GIT-stable or polystable) as a $(2,2)$-divisor in $\IP^1 \times \IP^1$. 
\end{thm}
\begin{proof}
We have shown that if $C$ is GIT-stable (that is, smooth), then $(X_0,\xi_0)$ is weighted K-polystable. If $C$ is GIT-unstable, then $(X_0,\xi_0)$ is weighted K-unstable by Remark \ref{Remark: 2.23(a) (1,1)-blowup} and \ref{Remark: 2.23(a) (2,1)-blowup}. If $C$ is strictly GIT-semistable, then there is a one-parameter subgroup $\IG_m\seq \SL_2\times\SL_2$ degenerating $C$ to a GIT-polystable curve. Hence it suffices to show that $(X_0,\xi_0)$ is weighted K-polystable if $C$ is GIT-polystable. 

By Theorem \ref{Theorem: GIT of biconic curves} (c), we see that $C=l_1+l_1'+l_2+l_2'$ for different rulings $l_1,l_1'\in|\CO(1,0)|, l_2,l_2'\in|\CO(0,1)|$, or $C=C_1+C_2$ for some smooth $C_1,C_2\in |\CO(1,1)|$. In the first case, $X_0=\Bl_C Q_0$ is toric. Then one may compute the minimizer of $\BH$ on the toric polytope of $X_0$ and see that it is just $\xi_0$. Thus $(X_0,\xi_0)$ is weighted K-polystable by \cite[Theorem 5.1]{BLXZ23}. In the second case, $X_0=\Bl_C Q_0$ is of complexity one, that is, it admits a two-dimensional torus action. By the second term of Remark \ref{Remark: 2.23(a) (1,1)-blowup}, more precisely, by {\bf Case (B)} of the $S^g$-computation, we see that 
$$\delta^g_{\IG_m^2}(X_0;R_\bu) \ge \min\{1,1,\delta^g(E;W^E_\bu)\} $$
using the weighted Abban-Zhuang estimate, where $\delta^g(E;W^E_\bu)> \frac{1}{0.95725} > 1$ and $E$ is the exceptional line of blowing up $\tH$ at a point of $C_1\cap C_2$. Now a similar argument as in Theorem \ref{Theorem: K-ps of optimal degeneration of №2.23(a)} will imply that $(X_0,\xi_0)$ is weighted K-polystable. 
\end{proof}

\subsection{The $\ord_{\tH}$-degeneration of $X$ in №2.23(b)}  \label{sec:№2.23b Hu degeneration} 
At the end of this section, we make a remark that, for any $a>0$, the divisorial valuation $a\cdot \ord_{\tH}$ does not minimize the $\BH$-invariant for $X$ in №2.23(b). Let $X_0$ be the special degeneration induced by $\ord_{\tH}$ with an inducing $\IG_m$-action and let $\xi_0$ be the soliton candidate with respect to the $\IG_m$-action. Then $X_0$ is a cone over $\IP(1,1,2)$. The singularities of $X_0$ are worse than those in the case of №2.23(a). We will see that $(X_0,\xi_0)$ is not weighted K-semistable. 

The refinement of $R_\bu=R(X_0)$ by $\tH\cong \IP(1,1,2)$ is
\begin{eqnarray*}
W_{(1,s)} = 
\left\{ \begin{array}{ll}
H^0\Big(\IP(1,1,2), \CO(2+s)\Big) 
& -1\le s < 0,\\
H^0\Big(\IP(1,1,2), \CO(2-s)\Big) + s\cdot C
& 0\le s\le 2, 
\end{array} \right.
\end{eqnarray*} 
where $\CO(1)=\CO_Q(1)|_H$. Hence the DH-measure on $\BP=[-1,2]$ is 
\begin{eqnarray*}
\DH_\BP(\dif s) = 
\left\{ \begin{array}{ll}
(2+s)^2 & -1\le s\le 0,\\
(2-s)^2 & 0\le s\le 2, 
\end{array} \right. 
\end{eqnarray*} 
which is totally the same with №2.23(a). Hence $X_0$ admits the same soliton candidate $\xi_0$ as (\ref{Number: soliton candidate of №2.23(a)}). 

Let $l\seq \IP(1,1,2)$ be a ruling of the cone. Then a straightforward computation shows that $S^g(W_\bu; \ord_l)\approx 1.04275 >1$. Pulling back $l$ to $X_0$ we get a prime divisor $D$. Then 
$S^g(R_\bu; \ord_D)=S^g(W_\bu; \ord_l) >1. $
Hence $(X_0,\xi_0)$ is weighted K-unstable.

\section{Optimal degenerations of Fano threefolds in family №2.23(b)}
\label{Section: Optimal deg of 2.23(b)}

We find the optimal degenerations of Fano threefolds $X$ in №2.23(b) in this section. Recall that $X=\Bl_C Q$ for the quadric threefold $Q$ and an elliptic curve $C=H\cap Q'$. We may assume that 
$$Q = \{uw + xy-z^2 = 0\}\seq \IP^4_{u,x,y,z,w}, \quad
H = \{ u = 0 \}|_Q, \quad
Q' = \{u\cdot l' + q = 0\}|_Q, $$ 
where $l'=l'(u,x,y,z,w)$ is linear and $q=q(x,y,z,w)$ is a quadric with $q(0,0,0,1)\ne 0$. Then $H=\{xy-z^2=0\} \seq \IP^3_{x,y,z,w}$ is the cone over conic. Let $o=[0,0,0,0,1]\in Q$ be the vertex of $H$. The above assumption on $q$ says that $o\notin C$. We will consider the degeneration of $C$ induced by the natural projection from $o$. More precisely, let $E_o$ be the exceptional divisor of the ordinary blowup of $X$ at $o$, that is, $\Bl_o X\to X$. Then 
$$\ord_{E_o}(u,x,y,z,w)=(2,1,1,1,0). $$
Hence the corresponding special degeneration $(\CX,\eta)$ of $X$ is induced by the $\IG_m$-action
$$\lam(t)\cdot[u,x,y,z,w]
=[t^{-2}u,t^{-1}x,t^{-1} y,t^{-1}z,w]. $$
Note that $Q$ and $H$ are invariant under this $\IG_m$-action, and 
$$\mathop{\lim}_{t\to 0} \lam(t)^*q(x,y,z,w)= \mathop{\lim}_{t\to 0}q(tx,ty,tz,w) = a_{44}\cdot w^2$$ 
for some $a_{44}\ne 0$. We see that the central fiber $X_0 = \Bl_{C_0} Q$ where $C_0=\{u=w^2=0\}|_Q$ is a non-reduced plane conic. We denote by $E_{C_0}$ the exceptional divisor. 

\subsection{Okounkov body and Zariski decomposition}

We are going to compute the refinement of $R_\bu=R(X_0,-K_{X_0})$ by $E_o$. Firstly, we construct an Okounkov body that is compatible with the $\IG_m$-action corresponding to $E_o$. Then we discuss the Zariski decomposition of $-K_{X_0}-tE_o$ which is more complicated than previous cases. Let $H_0=H-E_{C_0}$ be the strict transform of $H$ to $X_0$, and $\tH_0=H_0-E_o$ be the strict transform of $H_0$ to $\Bl_o X_0$. We denote by $\tH=H-E_o=\tH_0+E_{C_0}$ on $\Bl_o X_0$. 

Now we construct the Okounkov body. Recall that $R(X_0, -K_{X_0})$ is generated by 
$$H^0(X_0,-K_{X_0}) 
= u\cdot \la u,x,y,z,w \ra^2\oplus 
  w^2\cdot \la u,x,y,z,w \ra. $$
To get an Okounkov body compatible with the $\IG_m$-action corresponding to $E_o$, we need to take an admissible flag near $E_o$. Since $X_0=\Bl_{(u,w^2)}Q$, where $Q=\{uw+xy-z^2=0\}\seq \IP^4_{u,x,y,z,w}$. Restricting to the affine chart $U_w=\{w\ne0\}$ we have $X_0=Q=\IA^3_{x,y,z}$ with $u=z^2-xy$ and $\Bl_o X_0=\Bl_o \IA^3_{x,y,z} \seq \IA^3\times\IP^2_{\zeta_0,\zeta_1,\zeta_2}$ defined by $x:y:z=\zeta_0:\zeta_1:\zeta_2$. Further restricting to the affine chart $\{\zeta_0\ne 0\}$ and letting $y_1=\zeta_1/\zeta_0, z_1=\zeta_2/\zeta_0$, we get 
$$\tau: \Bl_o X_0 = \IA^3_{x,y_1,z_1} \to \IA^3_{x,y,z} = X_0, \quad (x,y_1,z_1) \mapsto (x, xy_1,xz_1). $$ 
Hence $E_o=\{x=0\}$ and $u=z^2-xy=x^2(z_1^2-y_1)$. We denote by $q=z_1^2-y_1$, then $\tH_0=\{q=0\}$. In this coordinate chart, we have 
\begin{eqnarray*}
H^0(X_0,-K_{X_0}) 
&=& u\cdot \la u,x,y,z,1 \ra^2\oplus 
           \la u,x,y,z,1 \ra\\
&=& x^2q\cdot \la x^2q,x,xy_1,xz_1,1 \ra^2\oplus 
              \la x^2q,x,xy_1,xz_1,1 \ra \\
&=& x^6q^3\oplus 
    x^5q^2 \cdot \la 1,y_1,z_1 \ra \oplus
    x^4q^2 \oplus x^4q \cdot \la 1,y_1,z_1 \ra^2 \oplus\\
& & x^3q \cdot \la 1,y_1,z_1 \ra \oplus
    x^2q \oplus
    x    \cdot \la 1,y_1,z_1 \ra \oplus
    1. 
\end{eqnarray*}
We choose the admissible flag $\IA^3_{x,y_1,z_1}\supseteq \{x=0\}\supseteq \{x=q=0\}\supseteq \{x=q=z_1=0\}$, which induces a faithful valuation $\fv: x\mapsto (1,0,0), q\mapsto (0,1,0), z_1\mapsto (0,0,1)$. We have $\fv(y_1)=\fv(z_1^2-q)=(0,0,2)$. Hence $\fv(H^0(X_0,-K_{X_0}))$ is given by 
\begin{eqnarray*}
x^6q^3 
&\mapsto& (6,3,0), \\
x^5q^2 \cdot \la 1,y_1,z_1 \ra  
&\mapsto& (5,2,0), (5,2,1), (5,2,2), \\
x^4q^2 \oplus x^4q \cdot \la 1,y_1,z_1 \ra^2  
&\mapsto& (4,2,0), (4,1,0), (4,1,0), (4,1,2), (4,1,3), (4,1,4), \\
x^3q \cdot \la 1,y_1,z_1 \ra 
&\mapsto& (3,1,0), (3,1,1), (3,1,2), \\
x^2q 
&\mapsto& (2,1,0), \\
x \cdot \la 1,y_1,z_1 \ra 
&\mapsto& (1,0,0), (1,0,1), (1,0,2),  \\
1
&\mapsto& (0,0,0).
\end{eqnarray*}
The convex hull of these points is an Okounkov body of $X_0$ since $R_\bu$ is generated by $H^0(X_0,-K_{X_0}) $. 
To make the Okounkov body compatible with the $\IG_m$-action, we may shift the first coordinate by $A_{X_0}(E_o)=3$, and get a polytope $\BO$ which lives in $-3\le x\le 3$. 
This Okounkov body $\BO$ is spanned by vertices 
$$(-3,0,0),(-2,0,0),(-2,0,2),(1,1,0),(1,1,4),(3,3,0). $$
Hence the Duistermaat-Heckman measure on $\BP=[-3,3]$ is 
\begin{eqnarray*}
\DH_\BP(\dif x) = 
\left\{ \begin{array}{ll}
\frac{1}{2}(3+x)^2 & -3\le x\le -2,\\
\frac{1}{18}(5+x)^2 & -2\le x\le 1,\\
\frac{1}{2}(3-x)^2 & 1\le x\le 3. 
\end{array} \right. 
\end{eqnarray*} 
Recall that $\BH(\eta)=\log\big(\int_\BP e^{-\eta x} \cdot \DH_\BP(\dif x)\big)$. Again, the minimizer $\eta_0$ should satisfies $\nabla_{\eta} H(\eta_0)=-\TZ_{g_{\eta_0}}(\eta)=0$, so we get $\int_{-3}^{-2} x\cdot e^{-\eta_0 x}\cdot\frac{1}{2}(3+x)^2\, \dif x+\int_{-2}^{1} x\cdot e^{-\eta_0 x}\cdot\frac{1}{18}(5+x)^2\, \dif x +\int_1^3 x\cdot e^{-\eta_0 x}\cdot\frac{1}{2}(3-x)^2 \, \dif x=0$. We can numerically solve the equation and obtain  $$\eta_0 \approx  0.15464273351085497. $$

In order to compute the refinement $W_\bu$ of $R_\bu$ by $E_o$, we firstly need to compute the Zariski decomposition of $-K_{X_0}-tE_o$ for $t\ge 0$. Motivated by the Okounkov body above, there are three cases of the decomposition, that is, $t$ belongs to $[0,1], [1,4]$ and $[4,6]$ respectively. Let $l\seq \tH_0$ be a line whose image in $H$ is a ruling through the vertex. 

When $0<t<1$, it's clear that $(-K_{X_0}-tE_o)$ is ample. Since $(-K_{X_0}-E_o)\cdot l=0$, we see that $(-K_{X_0}-E_o)$ is semi-ample. Furthermore, we have the following Zariski decomposition
$$-K_{X_0} - 4E_o = 2(\CO(1)-E_o) + \tH_0, $$
where $\CO(1)-E_o$ is semi-ample and $(\CO(1)-E_o)\cdot l = 0$. We conclude that for $1<t<4$, 
\begin{eqnarray*}
-K_{X_o}-tE_o 
&=& \frac{4-t}{3}\cdot(-K_{X_0}-E_o) 
  + \Big(1-\frac{4-t}{3}\Big)\cdot(-K_{X_0} - 4E_o) \\
&=& P_1(t) + \frac{1}{3}(t-1)\tH_0, 
\end{eqnarray*} 
where 
\begin{eqnarray*}
P_1(t)
&=& \frac{1}{3}(4-t)\cdot(-K_{X_0}-E_o) 
  + 2(1-\frac{1}{3}(4-t))\cdot(\CO(1)-E_o)\\
&=& \frac{1}{3}(10-t)\CO(1)
   -\frac{1}{3}(4-t)E_{C_0}
   -\frac{1}{3}(2+t)E_o
\end{eqnarray*}
is semi-ample. Finally, $-K_{X_0}-6E_o = \tH_0+2\tH$ is fixed. We have for $4<t<6$, 
\begin{eqnarray*}
-K_{X_0}-tE_o 
&=& \frac{6-t}{2}\cdot(-K_{X_0}-4E_o) 
  + \Big(1-\frac{6-t}{2}\Big)\cdot(-K_{X_0} - 6E_o) \\
&=& (6-t)(\CO(1)-E_o) + \tH_0+(t-4)\tH. 
\end{eqnarray*} 
Let $x=t-3$. Restricting to $E_o$, we get the $\IN^2$-graded linear series $W_\bu$, 
\begin{eqnarray*}
W_{(1,x)} = 
\left\{ \begin{array}{ll}
H^0\Big(\IP^2, \CO(3+x)\Big) 
& -3\le x < -2,\\
H^0\Big(\IP^2, \CO(\frac{1}{3}(5+x))\Big) + \frac{1}{3}(2+x)\cdot C_2
& -2\le x\le 1, \\
H^0\Big(\IP^2, \CO(3-x)\Big) + x\cdot C_2
& 1\le x\le 3, 
\end{array} \right. 
\end{eqnarray*} 
where $C_2=\tH|_{E_o}=\tH_0|_{E_o}$ is a smooth plane conic.

\subsection{Computing $S^g$ for the degeneration $X_0$ of №2.23(b)}\label{Section: compute S^g of №2.23(b)}
In this subsection, we show that $(X_0,\eta_0)$ is weighted K-polystable, hence it is the optimal degeneration of $X$ in №2.23(b). The proof is similar to the case in №2.23(a) when $C\seq \IP^1\times \IP^1$ is GIT-polystable. 

\begin{thm}
\label{Theorem: K-ps of optimal degeneration of №2.23(b)}
The pair $(X_0, \eta_0)$ is weighted K-polystable. 
\end{thm}

\begin{proof}
For $p\in \IP^2 \setminus C_2$, one may take refinement of $W_\bu$ by a line through $p$ and shows that $\delta_p(\IP^2;W_\bu)>1$. The key question is when $p\in C_2$, one may show that $\delta_p(\IP^2;W_\bu)=1$ (which will be shown latter). We need to take a further step refinement. 
Let $L\seq \IP^2$ be the line tangent to $C_2$ at $p$. Without loss of generality, just assume that $p=\{y_1=z_1=0\}$. Recall that $C_2 = \{z_1^2-y_1=0\}$, hence $L=\{y_1=0\}$. 
Let $E$ be the exceptional line of the $(2,1)$-blowup of $\IP^2$ at $p$ such that $\ord_E(y_1)=2, \ord_E(z_1)=1$. 

For $a>0,b\ge 0$, we define the linear series $V_{a,b}=H^0(\IP^2,\CO(a))+bC_2$. Then $V_{a,b}-tE$ admits Zariski decomposition
\begin{eqnarray*}
V_{a,b}-tE = 
\left\{ \begin{array}{ll}
H^0(\tilde{\IP}^2, \CO(a))+b\tC+(2b-t)E
& 0\le t \le 2b, \\
H^0(\tilde{\IP}^2, \CO(a)-(t-2b)E) + b\tC_2
& 2b\le t\le a+2b, \\
H^0(\tilde{\IP}^2, \CO(a)-(t-2b)E-(t-a-2b)\tL) + &\\
b\tC_2+(t-a-2b)\tL
& a+2b\le t\le 2a+2b.  
\end{array} \right. 
\end{eqnarray*}
Hence the volumes follows
\begin{eqnarray*}
\vol(V_{a,b}-tE) = 
\left\{ \begin{array}{ll}
a^2
& 0\le t \le 2b, \\
a^2-\frac{1}{2}(t-2b)^2
& 2b\le t\le a+2b, \\
a^2-\frac{1}{2}(t-2b)^2 + (t-a-2b)^2
& a+2b\le t\le 2a+2b.  
\end{array} \right. 
\end{eqnarray*}
Note that 
\begin{eqnarray*}
W_{(1,x)} = 
\left\{ \begin{array}{ll}
V_{3+x,0}
& -3\le x < -2,\\
V_{\frac{1}{3}(5+x), \frac{1}{3}(2+x)}
& -2\le x\le 1, \\
V_{(3-x), x}
& 1\le x\le 3. 
\end{array} \right. 
\end{eqnarray*} 
So we have 
\begin{eqnarray*} 
S^g(W_\bu; E) 
&=& \frac{1}{2\Bv^g} 
\Big(
\int_{-3}^{-2} g(x) 
\int_0^{6+2x} \vol(W_{(1,x)}-tE) 
\dif t 
\dif x  \\
&& \quad +  
\int_{-2}^1 g(x) 
\int_0^{\frac{2}{3}(7+2x)} \vol(W_{(1,x)}-tE) 
\dif t 
\dif x  \\
&& \quad +  
\int_{1}^3 g(x) 
\int_0^{6-2x} \vol(W_{(1,x)}-tE) 
\dif t 
\dif x  
\Big)\\
&=& 3 \,\,\,=\,\,\, A_{\IP^2}(E). 
\end{eqnarray*}
Hence $A_{\IP^2}(E)/S^g(W_\bu; E)=1$. 

Now we take the second step refinement. Recall that the $(2,1)$-blowup $\tilde{\IP^2}$ of $\IP^2$ has one ordinary double point $p_0\in E$. Then the different $\Diff_E(0)=\frac{1}{2}p_0$. Hence the linear series will be refined to $(\IP^1,\frac{1}{2}p_0)$. Let $W^E_\bu, V^E_{a,b,t}$ be the refinement of $W_\bu, V_{a,b}$ by $E$ respectively. Then 
\begin{eqnarray*}
V^E_{a,b,t} = 
\left\{ \begin{array}{ll}
H^0(\IP^1, \CO(\frac{1}{2}(t-2b))) +bp_2
& 2b\le t\le a+2b, \\
H^0(\IP^1, \CO(\frac{1}{2}(2a+2b-t)))+ &\\
bp_2+(t-a-2b)p_1
& a+2b\le t\le 2a+2b.  
\end{array} \right. 
\end{eqnarray*}
\begin{eqnarray*}
W^E_{(1,x,t)} = 
\left\{ \begin{array}{ll}
V^E_{3+x,0,t}
& -3\le x < -2,\\
V^E_{\frac{1}{3}(5+x), \frac{1}{3}(2+x),t}
& -2\le x\le 1, \\
V^E_{(3-x), x,t}
& 1\le x\le 3. 
\end{array} \right. 
\end{eqnarray*} 
A direct computation shows that, for $p\ne p_1,p_2$, 
\begin{eqnarray*}
S^g(W^E_\bu;p) 
&=& \frac{1}{12\Bv^g}
\Big( \int_{-3}^{-2} g(x) (3+x)^3 \dif x +
      \int_{-2}^{1} g(x) (\frac{1}{3}(5+x))^3 \dif x \\
&&
      + \int_{1}^{3} g(x) (3-x)^3 \dif x 
\Big) 
\,\,\,\approx\,\,\, 0.253588,
\\
S^g(W^E_\bu;p_1) 
&=& 2S^g(W^E_\bu;p) 
\,\,\,\approx\,\,\, 0.507176,
\\ 
S^g(W^E_\bu;p_2) 
&=& S^g(W^E_\bu;p) + 
\frac{1}{2\Bv^g}
\Big(\int_{-2}^{1} g(x) (\frac{1}{3}(5+x))^2 \cdot (\frac{1}{3}(2+x)) \dif x \\&&
+\int_{1}^{3} g(x) (3-x)^2 \cdot x \dif x 
\Big)
\,\,\,\approx\,\,\, 0.992824.
\end{eqnarray*} 
Hence
\begin{eqnarray*}
\delta^g(\IP^1,\frac{1}{2}p_0;W_\bu^E)
=\min\Big\{\frac{1}{2 S^g(W^E_\bu;p_0)}, 
\frac{1}{S^g(W^E_\bu;p_1)}, 
\frac{1}{S^g(W^E_\bu;p_2)}\Big\}
\approx \frac{1}{0.992824} > 1. 
\end{eqnarray*}

By the weighted Abban-Zhuang estimate, we have 
$$\delta^g_{\IG_m^2}(X_0;R_\bu) 
\ge \min\{1,1,\delta^g(\IP^1,\frac{1}{2}p_0;W_\bu^E)\}. $$
Since every point of $E$ is invariant under the $\IG_m^2$-action on $X_0$, we see that $\Val^{\IG_m^2,\circ}_{X_0} \cong \Val^\circ_{E}\times N_\IR(\IG_m^2)$. 
Now following the same argument of Theorem \ref{Theorem: K-ps of optimal degeneration of №2.23(a)}, we conclude that $(X_0,\eta_0)$ is weighted K-polystable. 
\end{proof}

\section{The delta invariants of Fano threefolds in family №2.23}

In this section, we find the delta minimizer of Fano threefolds $X$ in family №2.23. Recall that $X=\Bl_C Q$ is the blowup of the quadric threefold $Q$ along a smooth elliptic curve $C=H\cap Q'$, where $H\in |\CO_Q(1)|, Q'\in |\CO_Q(2)|$. Let $\tH$ be the strict transform of $H$ to $X$. We will show that $\delta(X)$ is minimized by the divisorial valuation $\ord_{\tH}$. 

We first compute the $S$-invariant of $\ord_{\tH}$ using intersection numbers on $X$. The anti-canonical divisor of $X$ is given by 
$-K_X = 3\pi^*H -E_C, $
where $E_C$ is the exceptional divisor of $\pi:X\to Q$. 
Then the Zariski decomposition of $-K_X - s\tH=P(s)+N(s)$ is as follows, 
    \begin{eqnarray*}
        P(s) = 
        \left\{ \begin{array}{ll}
        (3-s)\cdot \pi^*H + (s-1)\cdot E_C
        & 0\le s < 1,\\
        (3-s)\cdot \pi^*H 
        & 1\le s\le 3,
        \end{array} \right. 
     \quad \text{and}\quad 
        N(s) = 
        \left\{ \begin{array}{ll}
       0
        & 0\le s < 1,\\
        (s-1)\cdot E_C
        & 1\le s\le 3. 
        \end{array} \right. 
    \end{eqnarray*}
    Therefore, the volume is given by
    \begin{eqnarray*}
        \vol(-K_X - s\tH) = (P(s))^3 = 
        \left\{ \begin{array}{ll}
            -2(s^3+3s^2+3s-15) 
        & 0\le s < 1,\\
        2(3-s)^3
        & 1\le s\le 3. 
        \end{array} \right. 
    \end{eqnarray*}
    Hence, we get
    $$
    S_X(\tH) = \frac{1}{(-K_X)^3}\int_0^3 \text{vol}(-K_X-s\tH)\dif s = \frac{13}{12}.
    $$

\subsection{Local delta invariants of $X$ in №2.23(a)}
\begin{prop} \label{prop: local delta №2.23a}
    Let $X$ be a Fano threefold in №2.23(a). For any $p\in\tH$, the local delta invariant $\delta_p(X) = \frac{12}{13}$ is minimized by the divisorial valuation $\ord_{\tH}$.
\end{prop}

\begin{proof}
    Let $W^{\tH}_\bu$ be the refinement of $R_\bu=R(X,-K_X)$ given in Section \ref{Section: compute S^g of №2.23(a)}. It suffices to show $\delta_p(W^{\tH}_{\bu};\tH) \geq \frac{12}{13}$ for any $p \in \tH$. The following computation is almost the same with those in Section \ref{Section: compute S^g of №2.23(a)} (just replacing the weight function $g(x) = e^x$ by $g=1$). We only list the results. 

    \textbf{Case (A).} Suppose that $p \notin C$. Then let $l=l_1$ be one of the rulings of $\IP^1\times \IP^1$. We have $S(W^{\tH}_{\bu};l) =\frac{31}{40}$. Next, let $W^l_\bu$ be the refinement of $W_{\bu}^{\tH}$ by $l$. Then $S(W^l_{\bu};q)=\frac{31}{40}$ for any $q\in l$. Therefore, by the Abban-Zhuang estimate, we conclude that $\delta_p(W^{\tH}_{\bu}) \geq \min\{\frac{1}{S(W^{\tH}_{\bu};l)},\frac{1}{S(W_{\bu}^{l};q)}\}=\frac{40}{31}> \frac{12}{13}.$

\textbf{Case (B).} Suppose that $p \in C$ and the two rulings $l_1, l_2$ of $\IP^1\times \IP^1$ are not tangent to $C$ at $p$. Then let $\eta:\Tilde{Y}:=\Bl_p(\IP^1 \times \IP^1) \rightarrow \IP^1 \times \IP^1$ be the ordinary blowup of $p$ and we denote by $E$ the exceptional line. We have $S(W^{\tH}_{\bu};E)= \frac{109}{60}$. Next, let $W^E_\bu$ be the refinement of $W^{\tH}_\bu$ by $E$. We denote by $p_1 = \Tilde{l_1}\cap E, p_2 = \Tilde{l_2}\cap E, p_3 = \Tilde{C}\cap E$, which are pairwise distinct by assumption. We have
$$
S(W^{E}_{\bu}; p_0) = \frac{31}{60}, \,\,
S(W^{E}_{\bu}; p_1)=S(W^{E}_{\bu}; p_2)=\frac{31}{40},\,\, 
S(W^{E}_{\bu}; p_3)=\frac{47}{60} \,\,
$$ 
for any $p_0\in E$ different from $p_1,p_2$ and $p_3$. Hence, 
$$\delta(E;W^E_{\bu}) = \min\{\frac{1}{S(W^{E}_{\bu}; p )}, \frac{1}{S(W^{E}_{\bu}; p_1 )},\frac{1}{S(W^{E}_{\bu}; p_2 )},\frac{1}{S(W^{E}_{\bu}; p_3 )}\}= \frac{60}{47}>\frac{12}{13}.$$
We conclude that $\delta_p(W^{\tH}_{\bu};\tH) \geq \min\{\frac{A_{\tH}(E)}{S(W_{\bu}^{\tH};E)}, \delta_p(E;W^E_{\bu})\} = \frac{120}{109}> \frac{12}{13}.$

\textbf{Case (C).}  Suppose that $p \in C$, and one of the rulings $l_1, l_2$ of $\IP^1 \times \IP^1$ is tangent to $C$ at $p$. Using the same notions as {\bf Case (B)}, assume that $l_1$ is tangent to $C$ at $p$. So $p_1 = p_3$. Then 
$$S(W^E_{\bu};p_0)= \frac{31}{60}, \,\,
S(W^E_{\bu};p_1)=\frac{25}{24}, \,\,
S(W^E_{\bu};p_2) = \frac{31}{40}$$
for any $p_0\in E$ different from $p_1$ and $p_2$.
Hence,
$$\delta(E;W^E_{\bu}) = \min\{\frac{1}{S(W^{E}_{\bu}; p_0)}, \frac{1}{S(W^{E}_{\bu}; p_1 )},\frac{1}{S(W^{E}_{\bu}; p_2 )}\}= \frac{24}{25}>\frac{12}{13}.$$
We conclude that $\delta_p(W^{\tH}_{\bu};\tH) \geq \min\{\frac{A_{\tH}(E)}{S(W_{\bu}^{\tH};E)}, \delta_p(E;W^E_{\bu})\} = \frac{24}{25}> \frac{12}{13}.$

Combining {\bf Case (A), (B)} and {\bf (C)}, we see that
$$
\frac{12}{13} = \frac{1}{S_X(\tH)}\geq \delta_p(X) \geq \min\{ \frac{12}{13}, \frac{40}{31}, \frac{120}{109}, \frac{24}{25}\} = \frac{12}{13}. 
$$
Hence $\delta_p(X) = \frac{12}{13}$ is minimized by the divisorial valuation $\ord_{\tH}$.
\end{proof}

\subsection{Local delta invariants of $X$ in №2.23(b)}
In the case of №2.23(b), We need to work over $\tH\cong \IP(1,1,2)$, which is a cone over conic, instead of $\tH \cong \IP^1\times \IP^1$. We always denote by $l$ a ruling of the cone $\tH$. We see that $2l \in \CO_{\tH}(1)$. 

\begin{prop} 
\label{prop: local delta №2.23b}
    Let $X$ be a Fano threefold in №2.23(b). For any $p\in\tH$, the local delta invariant $\delta_p(X) = \frac{12}{13}$ is minimized by the divisorial valuation $\ord_{\tH}$.
\end{prop}
\begin{proof}
    Let $W^{\tH}_\bu$ be the refinement of $R_\bu=R(X,-K_X)$ by $\tH$, which is given by
\begin{eqnarray*}
    W_{(1,s)}^{\tH} = 
    \left\{ \begin{array}{ll}
    H^0\Big(\IP(1,1,2), \CO(1+s)\Big) 
    & 0\le s < 1,\\
    H^0\Big(\IP(1,1,2), \CO(3-s)\Big) +(s-1)\cdot C
    & 1\le s\le 3,
    \end{array} \right. 
\end{eqnarray*} 
where $C = \tH \cap E_C$. 

\textbf{Case (A).} Suppose that $p \notin C$. Let $l$ be a ruling of $\tH$ passing through $p$. The filtration induced by $l$ is
\begin{eqnarray*}
    \CF_{l}^{(t)} W_{(1,s)} = 
    \left\{ \begin{array}{ll}
    H^0\Big(\IP(1,1,2), \CO(1+s-\frac{t}{2})\Big)+t\cdot l
    & 0\le s < 1, 0\leq t \leq 2+2s,\\
    H^0\Big(\IP(1,1,2), \CO(3-s-\frac{t}{2})\Big)+ t\cdot l&\\
    + (s-1)\cdot C
    & 1\le s\le 3, 0\leq t \leq 6-2s. 
    \end{array} \right. 
\end{eqnarray*}
Hence
$$
S(W^{\tH}_{\bu};l) = \frac{1}{2\cdot 5}\cdot \left(\int_0^1 \int_0^{2+2s} 2\cdot \left(1+s-\frac{t}{2}\right)^2 \dif t + \int_1^3 \int_0^{6-2s} 2\cdot \left(3-s-\frac{t}{2}\right)^2 \dif t\right) = \frac{31}{30}< \frac{13}{12}.
$$

Next, let $W^l_\bu$ be the refinement of $W^{\tH}_{\bu}$ by $l$, which is given by
\begin{eqnarray*}
    W_{(1,s,t)}^{l} = 
    \left\{ \begin{array}{ll}
    H^0\Big(\IP^1, \CO(1+s-\frac{t}{2})\Big) 
    & 0\le s < 1, 0\leq t \leq 2+2s, \\
    H^0\Big(\IP^1, \CO(3-s-\frac{t}{2})\Big) +(s-1)\cdot (p_1+p_2)
    & 1\le s\le 3, 0\leq t \leq 6-2s,
    \end{array} \right. 
\end{eqnarray*} 
where $\{p_1,p_2\}=C\cap l$. Note that $p\in l$ is different from $p_1,p_2$, we have 
\begin{eqnarray*}
    \CF_{p}^{(x)} W_{(1,s,t)}^{l'} = 
    \left\{ \begin{array}{ll}
    H^0\Big(\IP^1, \CO(1+s-\frac{t}{2}-x)\Big) +x\cdot p
    & 0\le s < 1, 0\leq t \leq 2+2s,\\
    H^0\Big(\IP^1, \CO(3-s-\frac{t}{2}-x)\Big) +x\cdot p&\\
    + (s-1)\cdot (p_1+p_2)
    & 1\le s\le 3, 0\leq t \leq 6-2s. 
    \end{array} \right. 
\end{eqnarray*}
Hence $S(W^{l}_{\bu}; p) = \frac{31}{60}$.
If $p \in l$ is the vertex of the cone $\tH$, then $(K_{\tH}+l)|_l = K_l+\D_l$ where $\D_l=\Diff_l(0)=\frac{1}{2}\cdot p$. We conclude by the Abban-Zhuang estimate that
$$
\delta_p(W_{\bu}^{\tH}) \geq \min \{\frac{1}{S(W_{\bu}^{\tH};l)}, \frac{1-(1/2)}{S(W_{\bu}^{l};p)}\} = \frac{30}{31}> \frac{12}{13}.
$$

\begin{rmk}\rm 
The case where $p$ serves as the vertex of $\tH=\IP(1,1,2)$ is the crucial step in the proof. Other cases are standard since they are similar to those of №2.23(a). 
\end{rmk}

\textbf{Case (B). } Suppose that $p \in C$ and the ruling  $l$ passing through $p$ intersects $C$ transversally at $p$. Let $Y'$ be the ordinary blowup of $\IP(1,1,2)$ at $p$ and $E\cong \IP^1$ be the exceptional line. The negative part in the Zariski decomposition of $\CF_{E}^{(t)}W_{(1,s)}^{\tH}$ is given by
\begin{eqnarray*}
    N(\CF_{E}^{(t)}W_{(1,s)}^{\tH}) = 
    \left\{ \begin{array}{ll}
    t\cdot E 
    & 0\leq s < 1, 0\leq t \leq 1+s,\\
    t\cdot E + 2\cdot(t-1-s)\cdot \Tilde{l}
    & 0\leq s < 1, 1+s\leq t \leq 2+2s,\\
    (s-1)\cdot E + (s-1)\cdot \Tilde{C} & 1\leq s \leq 3, 0\leq t \leq s-1,\\
    t\cdot E+(s-1)\cdot \Tilde{C} & 1\leq s \leq 3, s-1\leq t \leq 2,\\
    t\cdot E+(s-1)\cdot \Tilde{C}+ 2(t-2)\cdot \Tilde{l} 
    & 1\le s\le 3, 2 \leq t \leq 5-s.\\
    \end{array} \right. 
\end{eqnarray*}
Then the volume is as follows
\begin{eqnarray*}
    \vol(\CF_{E}^{(t)}W_{(1,s)}^{\tH}) = 
    \left\{ \begin{array}{ll}
    2\cdot(1+s)^2 -t^2 
    & 0\leq s < 1, 0\leq t \leq 1+s,\\
    (2+2s-t)^2
    & 0\leq s < 1, 1+s\leq t \leq 2+2s,\\
    2\cdot (3-s)^2 & 1\leq s \leq 3, 0\leq t \leq s-1,\\
    2\cdot(3-s)^2-(s-1-t)^2 & 1\leq s \leq 3, s-1\leq t \leq 2,\\
    (5-s-t)^2   
    & 1\le s\le 3, 2 \leq t \leq 5-s.\\
    \end{array} \right. 
\end{eqnarray*}
Hence $S(W_{\bu}^{\tH};E) =\frac{109}{60}$. We have $\frac{A_{\tH}(E)}{S(W_{\bu}^{\tH};E)} = \frac{120}{109}>\frac{12}{13}$.

Next, let $W^E_\bu$ be the refinement of $W^{\tH}_{\bu}$ by $E$, which is given by
\begin{eqnarray*}
    W_{(1,s,t)}^{E} = 
    \left\{ \begin{array}{ll}
    H^0\Big(\IP^1, \CO(t)\Big) 
    & 0\le s < 1, 0\leq t \leq 1+s,\\
    H^0\Big(\IP^1, \CO(2+2s-t)\Big) +2(t-1-s)\cdot q_1 
    & 0\le s < 1, 1+s\leq t \leq 2+2s,\\
    H^0\Big(\IP^1, \CO(t-s+1)\Big) +(s-1)\cdot q_2 
    & 1\le s\le 3, s-1 \leq t \leq 2,\\
    H^0\Big(\IP^1, \CO(5-s-t)\Big) +2(t-2)\cdot q_1&\\
    +(s-1)\cdot q_2
    & 1\le s\le 3, 2 \leq t \leq 5-s, 
    \end{array} \right. 
\end{eqnarray*}
where $q_1 = \Tilde{l}\cap E$ and $q_2 = \Tilde{C}\cap E$. We have $q_1\neq q_2$ by assumption. Then we have 
$$S(W^E_{\bu};q) = \frac{31}{60}, \,\,
S(W^E_{\bu};q_1) = \frac{31}{30}, \,\,
S(W^E_{\bu};q_2) = \frac{47}{60}.$$ 
Hence
\begin{eqnarray*}
\delta(E;W^E_{\bu})
&=&\min\{ \frac{1}{S(W^E_{\bu};q)}, \frac{1}{S(W^E_{\bu};q_1)}, \frac{1}{S(W^E_{\bu};q_2)} \} = \frac{30}{31}, \\
\delta_p(\tH; W_{\bu}^{\tH})
&\geq& 
\min\{\frac{A_{\tH}(E)}{S(W_{\bu}^{\tH};E)}, \delta_p(E;W_{\bu}^E) \} = \frac{30}{31}>\frac{12}{13}.
\end{eqnarray*}

\textbf{Case (C).} Suppose that $p \in C$ and $l$ is tangent to $C$ at $p$.
Choose local coordinate $(X,Y)$ of $p\in \IP(1,1,2)$ such that 
$l = \{X = 0\}$. Let $\tY$ be the $(2,1)$-blowup of $\IP(1,1,2)$ at $p$ such that $\ord_E(X)=2, \ord_E(Y)=1$ where $E$ is the exceptional line. 
We have $\pi^*l = \Tilde{l}+2E, \pi^*C = \Tilde{C}+2E$. 
Let $p_1=\Tilde{l}\cap E, p_2 = \Tilde{C}\cap E$ and $p_0$ be the unique singular point of $\tY$. Note the three points $p_0, p_1, p_2$ are distinct.  Moreover, the negative part in the Zariski decomposition of $\CF_{E}^{(t)}W_{(1,s)}^{\tH}$ is given by
\begin{eqnarray*}
    N(\CF_{E}^{(t)}W_{(1,s)}^{\tH}) = 
    \left\{ \begin{array}{ll}
    t\cdot E 
    & 0\leq s < 1, 0\leq t \leq 1+s,\\
    t\cdot E + \frac{2}{3}\cdot(t-1-s)\cdot \Tilde{l}
    & 0\leq s < 1, 1+s\leq t \leq 4+4s,\\
    (2s-2)\cdot E + (s-1)\cdot \Tilde{C} & 1\leq s \leq 3, 0\leq t \leq 2s-2,\\
    t\cdot E + (s-1)\cdot \Tilde{C} & 1\leq s \leq 3, 2s-2\leq t \leq s+1,\\
    t\cdot E+ (s-1)\cdot \Tilde{C}&\\
    + \frac{2}{3}\cdot(t-s-1)\cdot \Tilde{l} 
    & 1\le s\le 3, s+1 \leq t \leq 10-2s.\\
    \end{array} \right. 
\end{eqnarray*}
The volume is as follows 
\begin{eqnarray*}
    \vol(\CF_{E}^{(t)}W_{(1,s)}^{\tH}) = 
    \left\{ \begin{array}{ll}
    \frac{1}{2}\cdot(2+2s)^2 -\frac{1}{2} t^2 
    & 0\leq s < 1, 0\leq t \leq 1+s,\\
    \frac{1}{2}\cdot (\frac{8}{3} s-\frac{2}{3} t+\frac{8}{3})^2-\frac{1}{2}\cdot (\frac{1}{3}t-\frac{4}{3}s-\frac{4}{3})^2
    & 0\leq s < 1, 1+s\leq t \leq 4+4s,\\
    \frac{1}{2}\cdot (6-2s)^2 & 1\leq s \leq 3, 0\leq t \leq 2s-2,\\
    \frac{1}{2}\cdot (6-2s)^2-\frac{1}{2}\cdot(2s-2-t)^2 & 1\leq s \leq 3, 2s-2\leq t \leq s+1,\\
    \frac{1}{2}\cdot (\frac{20}{3}-\frac{4}{3}s-\frac{2}{3}t)^2-\frac{1}{2}\cdot(-\frac{10}{3}+\frac{2}{3}s+\frac{1}{3}t)^2   
    & 1\le s\le 3, s+1 \leq t \leq 10-2s.\\
    \end{array} \right. 
\end{eqnarray*}
Hence 
$S(W_{\bu}^{\tH}; E) = \frac{187}{60}. $
We get $\frac{A_{\tH}(E)}{S(W_{\bu}^{\tH};E)} = \frac{180}{187} >\frac{12}{13}$. 

Next, let $W^E_\bu$ be the refinement of $W^{\tH}_{\bu}$ by $E$, which is given by
\begin{eqnarray*}
    W_{(1,s,t)}^{E} = 
    \left\{ \begin{array}{ll}
    H^0\Big(\IP^1, \CO(\frac{t}{2})\Big) 
    & 0\le s < 1, 0\leq t \leq 1+s,\\
    H^0\Big(\IP^1, \CO(\frac{2}{3}+\frac{2}{3}s-\frac{1}{6}t)\Big) +\frac{2}{3}\cdot(t-1-s)\cdot p_1  
    & 0\le s < 1, 1+s\leq t \leq 4+4s,\\
    H^0\Big(\IP^1, \CO(\frac{t}{2}-s+1)\Big) +(s-1)\cdot p_2 
    & 1\le s\le 3, 2s-2 \leq t \leq s+1,\\
    H^0\Big(\IP^1, \CO(\frac{5}{3}-\frac{1}{3}s-\frac{1}{6}t)\Big) +(s-1)\cdot p_2&\\
    +\frac{2}{3}\cdot (t-s-1)\cdot p_1
    & 1\le s\le 3, s+1 \leq t \leq 10-2s. 
    \end{array} \right. 
\end{eqnarray*}
Hence 
$$S(W^E_{\bu};p_0) = \frac{31}{120}, \,\,
S(W^E_{\bu};p_1) = \frac{31}{30}, \,\,
S(W^E_{\bu};p_2) = \frac{21}{40}. $$ 
We conclude that 
\begin{eqnarray*}
\delta(E;W^E_{\bu})
&=&\min\{ \frac{1-(1/2)}{S(W^E_{\bu};p_0)}, \frac{1}{S(W^E_{\bu};p_1)}, \frac{1}{S(W^E_{\bu};p_2)} \} = \frac{30}{31}, \\
\delta_p(\tH; W_{\bu}^{\tH})
&\geq&
\min\{\frac{A_{\tH}(E)}{S(W_{\bu}^{\tH};E)}, \delta_p(E;W_{\bu}^E) \} = \frac{180}{187}>\frac{12}{13}.
\end{eqnarray*}

Combining {\bf Cases (A), (B)}, and {\bf (C)}, we have
$$
\frac{12}{13} = \frac{1}{S_X(\tH)}\geq \delta_p(X) \geq \min\{\frac{12}{13}, \frac{30}{31}, \frac{180}{187} \} = \frac{12}{13}. 
$$
Hence $\delta_p(X) = \frac{12}{13}$ is minimized by the divisorial valuation $\ord_{\tH}$.
\end{proof}

\subsection{The delta invariants of $X$ in №2.23}
\begin{thm}
    For any $X$ in №2.23, the delta invariant $\delta(X)= \frac{12}{13}$ is minimized by $\ord_{\tH}$.
\end{thm}
\begin{proof}
    For any $X$ in №2.23(a) 
    or №2.23(b), we utilize the one-parameter subgroup $$\lambda(t)\cdot [u,x,y,z,w]=[t^{-1}u,x,y,z,w]$$ 
    as in the beginning of Section \ref{Section: Optimal deg of 2.23(a)}, to produce a special degeneration $X = \text{Bl}_C Q \rightsquigarrow X_0=\Bl_{C} Q_0$. We claim that $\delta(X_0) = \frac{12}{13}$. 
    Then by the lower semicontinuity of the delta invariant \cite{BL18b}, we get $$\frac{12}{13}
    =\frac{A_X(\tH)}{S_X(\tH)}
    \geq\delta(X)
    \geq\delta(X_0)
    =\frac{12}{13}.$$ 
    Hence $\delta(X)=\frac{12}{13}$ is minimized by $\ord_{\tH}$. 

    Now we prove the claim. Note that $X_0$ admits a $\IG_m$-action. The $\IG_m$-invariant subsets of $X_0$ are just the vertex $p_0$ of the cone $Q_0$, and the subvarieties of $\tH$. Hence it suffices to show $\delta_p(X_0)\ge\frac{12}{13}$ for $p=p_0$ and any closed point $p\in \tH$ by \cite{Zhu21}. For $p=p_0$, one may take refinement by the exceptional divisor of the ordinary blowup of the point $p_0$ and shows that $\delta_{p_0}(X_0)>1$. For $p\in \tH$, we first see that $S_{X_0}(\tH)= \frac{13}{12}$ as computed in the beginning of this section. 
    Next, let $W^{\tH}_\bu$ be the refinement of $R_\bu=R(X_0,-K_{X_0})$ by $\tH$. Then $\delta_p(\tH; W^{\tH}_\bu) \ge \frac{12}{13}$ by Proposition \ref{prop: local delta №2.23a} and \ref{prop: local delta №2.23b}. 
    Hence $\delta(X_0)=\delta_{\IG_m}(X_0)=\frac{12}{13}$ is minimized by $\ord_{\tH}$. 
\end{proof}

\bibliographystyle{alpha}
\bibliography{ref}

@article {BL18b,
    AUTHOR = {Blum, Harold and Liu, Yuchen},
     TITLE = {Openness of uniform {K}-stability in families of {$\Bbb
              Q$}-{F}ano varieties},
   JOURNAL = {Ann. Sci. \'{E}c. Norm. Sup\'{e}r. (4)},
  FJOURNAL = {Annales Scientifiques de l'\'{E}cole Normale Sup\'{e}rieure. Quatri\`eme
              S\'{e}rie},
    VOLUME = {55},
      YEAR = {2022},
    NUMBER = {1},
     PAGES = {1--41},
      ISSN = {0012-9593},
   MRCLASS = {14J45 (14E07 32Q26)},
  MRNUMBER = {4411858},
MRREVIEWER = {Alexandr V. Pukhlikov},
       DOI = {10.24033/asens.2490},
       URL = {https://doi.org/10.24033/asens.2490},
}

@article {Tia97,
    AUTHOR = {Tian, Gang},
     TITLE = {K\"{a}hler-{E}instein metrics with positive scalar curvature},
   JOURNAL = {Invent. Math.},
  FJOURNAL = {Inventiones Mathematicae},
    VOLUME = {130},
      YEAR = {1997},
    NUMBER = {1},
     PAGES = {1--37},
      ISSN = {0020-9910},
   MRCLASS = {53C25 (32L07 53C55)},
  MRNUMBER = {1471884},
MRREVIEWER = {Thalia D. Jeffres},
       DOI = {10.1007/s002220050176},
       URL = {https://doi.org/10.1007/s002220050176},
}

@article {HL20,
    AUTHOR = {Han, Jiyuan and Li, Chi},
     TITLE = {Algebraic uniqueness of {K}\"{a}hler-{R}icci flow limits and
              optimal degenerations of {F}ano varieties},
   JOURNAL = {Geom. Topol.},
  FJOURNAL = {Geometry \& Topology},
    VOLUME = {28},
      YEAR = {2024},
    NUMBER = {2},
     PAGES = {539--592},
      ISSN = {1465-3060,1364-0380},
   MRCLASS = {14J45 (32Q26 53E30)},
  MRNUMBER = {4718122},
MRREVIEWER = {S\l awomir\ Dinew},
       DOI = {10.2140/gt.2024.28.539},
       URL = {https://doi.org/10.2140/gt.2024.28.539},
}

@article {BHJ17,
    AUTHOR = {Boucksom, S\'{e}bastien and Hisamoto, Tomoyuki and Jonsson,
              Mattias},
     TITLE = {Uniform {K}-stability, {D}uistermaat-{H}eckman measures and
              singularities of pairs},
   JOURNAL = {Ann. Inst. Fourier (Grenoble)},
  FJOURNAL = {Universit\'{e} de Grenoble. Annales de l'Institut Fourier},
    VOLUME = {67},
      YEAR = {2017},
    NUMBER = {2},
     PAGES = {743--841},
      ISSN = {0373-0956},
   MRCLASS = {32Q26 (14E30 14G22)},
  MRNUMBER = {3669511},
MRREVIEWER = {Yuji Odaka},
       URL = {http://aif.cedram.org/item?id=AIF_2017__67_2_743_0},
}

@article {JM12,
    AUTHOR = {Jonsson, Mattias and Musta\c{t}\u{a}, Mircea},
     TITLE = {Valuations and asymptotic invariants for sequences of ideals},
   JOURNAL = {Ann. Inst. Fourier (Grenoble)},
  FJOURNAL = {Universit\'{e} de Grenoble. Annales de l'Institut Fourier},
    VOLUME = {62},
      YEAR = {2012},
    NUMBER = {6},
     PAGES = {2145--2209 (2013)},
      ISSN = {0373-0956},
   MRCLASS = {14F18 (12J20 14B05)},
  MRNUMBER = {3060755},
MRREVIEWER = {Carlos Galindo},
       DOI = {10.5802/aif.2746},
       URL = {https://doi.org/10.5802/aif.2746},
}

@ARTICLE{AZ22,
    AUTHOR = {Abban, Hamid and Zhuang, Ziquan},
     TITLE = {K-stability of {F}ano varieties via admissible flags},
   JOURNAL = {Forum Math. Pi},
  FJOURNAL = {Forum of Mathematics. Pi},
    VOLUME = {10},
      YEAR = {2022},
     PAGES = {Paper No. e15, 43},
      ISSN = {2050-5086},
   MRCLASS = {14J45 (32Q20)},
  MRNUMBER = {4448177},
       DOI = {10.1017/fmp.2022.11},
       URL = {https://doi.org/10.1017/fmp.2022.11},
}

@ARTICLE{XZ19,
    AUTHOR = {Xu, Chenyang and Zhuang, Ziquan},
     TITLE = {On positivity of the {CM} line bundle on {K}-moduli spaces},
   JOURNAL = {Ann. of Math. (2)},
  FJOURNAL = {Annals of Mathematics. Second Series},
    VOLUME = {192},
      YEAR = {2020},
    NUMBER = {3},
     PAGES = {1005--1068},
}

@article {LX18,
    AUTHOR = {Li, Chi and Xu, Chenyang},
     TITLE = {Stability of {V}aluations: {H}igher {R}ational {R}ank},
   JOURNAL = {Peking Math. J.},
  FJOURNAL = {Peking Mathematical Journal},
    VOLUME = {1},
      YEAR = {2018},
    NUMBER = {1},
     PAGES = {1--79},
      ISSN = {2096-6075},
   MRCLASS = {14E30 (14B05 32Q26)},
  MRNUMBER = {4059992},
       DOI = {10.1007/s42543-018-0001-7},
       URL = {https://doi.org/10.1007/s42543-018-0001-7},
}

@book{ACC+,
    AUTHOR = {Araujo, Carolina and Castravet, Ana-Maria and Cheltsov, Ivan
              and Fujita, Kento and Kaloghiros, Anne-Sophie and
              Martinez-Garcia, Jesus and Shramov, Constantin and S\"{u}\ss ,
              Hendrik and Viswanathan, Nivedita},
     TITLE = {The {C}alabi problem for {F}ano threefolds},
    SERIES = {London Mathematical Society Lecture Note Series},
    VOLUME = {485},
 PUBLISHER = {Cambridge University Press, Cambridge},
      YEAR = {2023},
     PAGES = {vii+441},
      ISBN = {978-1-009-19339-9},
   MRCLASS = {14J45 (32Q15 32Q20)},
  MRNUMBER = {4590444},
}

@article {BJ20,
    AUTHOR = {Blum, Harold and Jonsson, Mattias},
     TITLE = {Thresholds, valuations, and {K}-stability},
   JOURNAL = {Adv. Math.},
  FJOURNAL = {Advances in Mathematics},
    VOLUME = {365},
      YEAR = {2020},
     PAGES = {107062, 57},
      ISSN = {0001-8708,1090-2082},
   MRCLASS = {14C20 (14M25)},
  MRNUMBER = {4067358},
MRREVIEWER = {Chenyang\ Xu},
       DOI = {10.1016/j.aim.2020.107062},
       URL = {https://doi.org/10.1016/j.aim.2020.107062},
}

@article {BLXZ23,
    AUTHOR = {Blum, Harold and Liu, Yuchen and Xu, Chenyang and Zhuang,
              Ziquan},
     TITLE = {The existence of the {K}\"{a}hler-{R}icci soliton
              degeneration},
   JOURNAL = {Forum Math. Pi},
  FJOURNAL = {Forum of Mathematics. Pi},
    VOLUME = {11},
      YEAR = {2023},
     PAGES = {Paper No. e9, 28},
      ISSN = {2050-5086},
   MRCLASS = {14J45 (14D06 14E99 32Q20)},
  MRNUMBER = {4558985},
       DOI = {10.1017/fmp.2023.5},
       URL = {https://doi.org/10.1017/fmp.2023.5},
}

@article {HL23,
    AUTHOR = {Han, Jiyuan and Li, Chi},
     TITLE = {On the {Y}au-{T}ian-{D}onaldson conjecture for generalized
              {K}\"{a}hler-{R}icci soliton equations},
   JOURNAL = {Comm. Pure Appl. Math.},
  FJOURNAL = {Communications on Pure and Applied Mathematics},
    VOLUME = {76},
      YEAR = {2023},
    NUMBER = {9},
     PAGES = {1793--1867},
      ISSN = {0010-3640,1097-0312},
   MRCLASS = {53C55 (32W20)},
  MRNUMBER = {4612573},
}

@article {LXZ22,
    AUTHOR = {Liu, Yuchen and Xu, Chenyang and Zhuang, Ziquan},
     TITLE = {Finite generation for valuations computing stability
              thresholds and applications to {K}-stability},
   JOURNAL = {Ann. of Math. (2)},
  FJOURNAL = {Annals of Mathematics. Second Series},
    VOLUME = {196},
      YEAR = {2022},
    NUMBER = {2},
     PAGES = {507--566},
      ISSN = {0003-486X,1939-8980},
   MRCLASS = {14J45 (14D20 14E99 32Q20)},
  MRNUMBER = {4445441},
       DOI = {10.4007/annals.2022.196.2.2},
       URL = {https://doi.org/10.4007/annals.2022.196.2.2},
}

@article {Zhu21,
    AUTHOR = {Zhuang, Ziquan},
     TITLE = {Optimal destabilizing centers and equivariant {K}-stability},
   JOURNAL = {Invent. Math.},
  FJOURNAL = {Inventiones Mathematicae},
    VOLUME = {226},
      YEAR = {2021},
    NUMBER = {1},
     PAGES = {195--223},
      ISSN = {0020-9910,1432-1297},
   MRCLASS = {14J45 (32Q26 53C55)},
  MRNUMBER = {4309493},
MRREVIEWER = {Usha\ N.\ Bhosle},
       DOI = {10.1007/s00222-021-01046-0},
       URL = {https://doi.org/10.1007/s00222-021-01046-0},
}

@article {RTZ21,
    AUTHOR = {Rubinstein, Yanir A. and Tian, Gang and Zhang, Kewei},
     TITLE = {Basis divisors and balanced metrics},
   JOURNAL = {J. Reine Angew. Math.},
  FJOURNAL = {Journal f\"{u}r die Reine und Angewandte Mathematik. [Crelle's
              Journal]},
    VOLUME = {778},
      YEAR = {2021},
     PAGES = {171--218},
      ISSN = {0075-4102,1435-5345},
   MRCLASS = {32Q26 (14J45 53C55)},
  MRNUMBER = {4308614},
MRREVIEWER = {Ahmed\ Lesfari},
       DOI = {10.1515/crelle-2021-0017},
       URL = {https://doi.org/10.1515/crelle-2021-0017},
}

@article {TZ02,
    AUTHOR = {Tian, Gang and Zhu, Xiaohua},
     TITLE = {A new holomorphic invariant and uniqueness of {K}\"{a}hler-{R}icci
              solitons},
   JOURNAL = {Comment. Math. Helv.},
  FJOURNAL = {Commentarii Mathematici Helvetici},
    VOLUME = {77},
      YEAR = {2002},
    NUMBER = {2},
     PAGES = {297--325},
      ISSN = {0010-2571},
   MRCLASS = {32Q20 (53C25 53C55)},
  MRNUMBER = {1915043},
MRREVIEWER = {Peng Lu},
       DOI = {10.1007/s00014-002-8341-3},
       URL = {https://doi.org/10.1007/s00014-002-8341-3},
}

@book {Xu23,
    AUTHOR = {Xu, Chenyang},
     TITLE = {K-stability of {F}ano varieties},
    SERIES = {New Mathematical Monographs},
    VOLUME = {50},
 PUBLISHER = {Cambridge University Press, Cambridge},
      YEAR = {2025},
     PAGES = {xi+411},
      ISBN = {978-1-009-53877-0; [9781009538763]},
   MRCLASS = {14J45 (32Q26)},
  MRNUMBER = {4893062},
}

@article {LL23,
    AUTHOR = {Li, Yan and Li, Zhenye},
     TITLE = {Equivariant {$\Bbb R$}-test configurations and semistable
              limits of {$\Bbb Q$}-{F}ano group compactifications},
   JOURNAL = {Peking Math. J.},
  FJOURNAL = {Peking Mathematical Journal},
    VOLUME = {6},
      YEAR = {2023},
    NUMBER = {2},
     PAGES = {559--607},
      ISSN = {2096-6075},
   MRCLASS = {53E30 (53C25 58D25)},
  MRNUMBER = {4619601},
       DOI = {10.1007/s42543-022-00054-0},
       URL = {https://doi.org/10.1007/s42543-022-00054-0},
}

@article {DS20,
    AUTHOR = {Dervan, Ruadha\'{\i} and Sz\'{e}kelyhidi, G\'{a}bor},
     TITLE = {The {K}\"{a}hler-{R}icci flow and optimal degenerations},
   JOURNAL = {J. Differential Geom.},
  FJOURNAL = {Journal of Differential Geometry},
    VOLUME = {116},
      YEAR = {2020},
    NUMBER = {1},
     PAGES = {187--203},
      ISSN = {0022-040X},
   MRCLASS = {53E30 (53C55)},
  MRNUMBER = {4146359},
MRREVIEWER = {Frederick Tsz-Ho Fong},
       DOI = {10.4310/jdg/1599271255},
       URL = {https://doi.org/10.4310/jdg/1599271255},
}

@article {TZZZ13,
    AUTHOR = {Tian, Gang and Zhang, Shijin and Zhang, Zhenlei and Zhu,
              Xiaohua},
     TITLE = {Perelman's entropy and {K}\"{a}hler-{R}icci flow on a {F}ano
              manifold},
   JOURNAL = {Trans. Amer. Math. Soc.},
  FJOURNAL = {Transactions of the American Mathematical Society},
    VOLUME = {365},
      YEAR = {2013},
    NUMBER = {12},
     PAGES = {6669--6695},
      ISSN = {0002-9947},
   MRCLASS = {53C25 (53C55 58J05)},
  MRNUMBER = {3105766},
       DOI = {10.1090/S0002-9947-2013-06027-8},
       URL = {https://doi.org/10.1090/S0002-9947-2013-06027-8},
}

@article {TZ07,
    AUTHOR = {Tian, Gang and Zhu, Xiaohua},
     TITLE = {Convergence of {K}\"{a}hler-{R}icci flow},
   JOURNAL = {J. Amer. Math. Soc.},
  FJOURNAL = {Journal of the American Mathematical Society},
    VOLUME = {20},
      YEAR = {2007},
    NUMBER = {3},
     PAGES = {675--699},
      ISSN = {0894-0347},
   MRCLASS = {53C44 (53C55)},
  MRNUMBER = {2291916},
MRREVIEWER = {Julien Keller},
       DOI = {10.1090/S0894-0347-06-00552-2},
       URL = {https://doi.org/10.1090/S0894-0347-06-00552-2},
}

@article {TZ15,
    AUTHOR = {Tian, Gang and Zhu, Xiaohua},
     TITLE = {Convergence of the {K}\"{a}hler-{R}icci flow on {F}ano manifolds},
   JOURNAL = {J. Reine Angew. Math.},
  FJOURNAL = {Journal f\"{u}r die Reine und Angewandte Mathematik. [Crelle's
              Journal]},
    VOLUME = {678},
      YEAR = {2013},
     PAGES = {223--245},
      ISSN = {0075-4102},
   MRCLASS = {32Q20 (53C44)},
  MRNUMBER = {3056108},
MRREVIEWER = {Julien Keller},
       DOI = {10.1515/crelle.2012.021},
       URL = {https://doi.org/10.1515/crelle.2012.021},
}

@article {LWX21,
    AUTHOR = {Li, Chi and Wang, Xiaowei and Xu, Chenyang},
     TITLE = {Algebraicity of the metric tangent cones and equivariant
              {K}-stability},
   JOURNAL = {J. Amer. Math. Soc.},
  FJOURNAL = {Journal of the American Mathematical Society},
    VOLUME = {34},
      YEAR = {2021},
    NUMBER = {4},
     PAGES = {1175--1214},
      ISSN = {0894-0347},
   MRCLASS = {14B07 (14E30 14J17 14J45 53C55)},
  MRNUMBER = {4301561},
MRREVIEWER = {Carlos Galindo},
       DOI = {10.1090/jams/974},
       URL = {https://doi.org/10.1090/jams/974},
}

@article {MW23,
    AUTHOR = {Miao, Minghao and Wang, Linsheng},
     TITLE = {K\"{a}hler-{R}icci solitons on {F}ano threefolds with
              non-trivial moduli},
   JOURNAL = {Math. Ann.},
  FJOURNAL = {Mathematische Annalen},
    VOLUME = {392},
      YEAR = {2025},
    NUMBER = {3},
     PAGES = {4483--4523},
      ISSN = {0025-5831,1432-1807},
   MRCLASS = {14J45 (14L24 32Q26 53C25 53C55)},
  MRNUMBER = {4939693},
MRREVIEWER = {S\l awomir\ Dinew},
       DOI = {10.1007/s00208-025-03202-w},
       URL = {https://doi.org/10.1007/s00208-025-03202-w},
}

@article {MT22,
    AUTHOR = {Miao, Minghao and Tian, Gang},
     TITLE = {A note on {K}\"{a}hler-{R}icci flow on {F}ano threefolds},
   JOURNAL = {Peking Math. J.},
  FJOURNAL = {Peking Mathematical Journal},
    VOLUME = {8},
      YEAR = {2025},
    NUMBER = {1},
     PAGES = {191--199},
      ISSN = {2096-6075,2524-7182},
   MRCLASS = {53E30 (32Q26)},
  MRNUMBER = {4870950},
MRREVIEWER = {Kai\ Zheng},
       DOI = {10.1007/s42543-023-00078-0},
       URL = {https://doi.org/10.1007/s42543-023-00078-0},
}

@article {TZ16,
    AUTHOR = {Tian, Gang and Zhang, Zhenlei},
     TITLE = {Regularity of {K}\"{a}hler-{R}icci flows on {F}ano manifolds},
   JOURNAL = {Acta Math.},
  FJOURNAL = {Acta Mathematica},
    VOLUME = {216},
      YEAR = {2016},
    NUMBER = {1},
     PAGES = {127--176},
      ISSN = {0001-5962},
   MRCLASS = {53C44 (32Q15 53C55)},
  MRNUMBER = {3508220},
MRREVIEWER = {Shouwen Fang},
       DOI = {10.1007/s11511-016-0137-1},
       URL = {https://doi.org/10.1007/s11511-016-0137-1},
}

@article {Bam18,
    AUTHOR = {Bamler, Richard},
     TITLE = {Convergence of {R}icci flows with bounded scalar curvature},
   JOURNAL = {Ann. of Math. (2)},
  FJOURNAL = {Annals of Mathematics. Second Series},
    VOLUME = {188},
      YEAR = {2018},
    NUMBER = {3},
     PAGES = {753--831},
      ISSN = {0003-486X},
   MRCLASS = {53C44 (53C23 53C56)},
  MRNUMBER = {3866886},
MRREVIEWER = {Chih-Wei Chen},
       DOI = {10.4007/annals.2018.188.3.2},
       URL = {https://doi.org/10.4007/annals.2018.188.3.2},
}

@article {CW20,
    AUTHOR = {Chen, Xiuxiong and Wang, Bing},
     TITLE = {Space of {R}icci flows ({II})---{P}art {B}: {W}eak compactness
              of the flows},
   JOURNAL = {J. Differential Geom.},
  FJOURNAL = {Journal of Differential Geometry},
    VOLUME = {116},
      YEAR = {2020},
    NUMBER = {1},
     PAGES = {1--123},
      ISSN = {0022-040X},
   MRCLASS = {53E30 (53C25 53C55 58D27)},
  MRNUMBER = {4146357},
MRREVIEWER = {Dmitri\u{\i} Vladimir Alekseevsky},
       DOI = {10.4310/jdg/1599271253},
       URL = {https://doi.org/10.4310/jdg/1599271253},
}

@article {WZ21,
    AUTHOR = {Wang, Feng and Zhu, Xiaohua},
     TITLE = {Tian's partial ${C}^0$-estimate implies {H}amilton-{T}ian's conjecture},
   JOURNAL = {Adv. Math.},
  FJOURNAL = {Advances in Mathematics},
    VOLUME = {381},
      YEAR = {2021},
     PAGES = {107619, 44},
      ISSN = {0001-8708, 1090-2082},
  MRNUMBER = {4206790},
       DOI = {10.1016/j.aim.2021.107619},
       URL = {https://doi.org/10.1016/j.aim.2021.107619},
}

@article{TZ22,
      title={Horosymmetric limits of {K}\"ahler-{R}icci flow on {F}ano ${G}$-manifolds}, 
      author={Gang Tian and Xiaohua Zhu}, 
      journal = {J. Eur. Math. Soc.},
      year={2024},
      note = {\href{https://ems.press/journals/jems/articles/14298245}{https://ems.press/journals/jems/articles/14298245}},
}

@article{TZ25,
  title={H-Invariant on {K}\"ahler Manifolds},
  author={Tian, Gang and Zhu, Xiaohua},
  journal={Peking Math. J.},
  year={2025},
  publisher={Springer},
  doi={10.1007/s42543-025-00108-z},
  note={\href{https://doi.org/10.1007/s42543-025-00108-z}{https://doi.org/10.1007/s42543-025-00108-z}}
}

@article{WZ23,
  title={K\"ahler-{R}icci flow on ${G}$-spherical {F}ano manifolds},
  author={Wang, Feng and Zhu, Xiaohua},
  year={2026},
  journal = {Peking Math. J.},
  note = {\href{https://doi.org/10.1007/s42543-024-00088-6}{https://doi.org/10.1007/s42543-024-00088-6}},
}

@article {CDF+,
    AUTHOR = {Cheltsov, Ivan and Duarte Guerreiro, Tiago and Fujita, Kento
              and Krylov, Igor and Martinez-Garcia, Jesus},
     TITLE = {K-stability of {C}asagrande-{D}ruel varieties},
   JOURNAL = {J. Reine Angew. Math.},
  FJOURNAL = {Journal f\"{u}r die Reine und Angewandte Mathematik. [Crelle's
              Journal]},
    VOLUME = {818},
      YEAR = {2025},
     PAGES = {53--113},
      ISSN = {0075-4102},
   MRCLASS = {14J45 (14D20 14E30)},
  MRNUMBER = {4846021},
       DOI = {10.1515/crelle-2024-0074},
       URL = {https://doi.org/10.1515/crelle-2024-0074},
}

@article{Wang24,
      title={A valuative criterion of {K}-polystability}, 
      author={Linsheng Wang},
      year={2024},
      JOURNAL = {Trans. Amer. Math. Soc., to appear},
      note = {\href{https://arxiv.org/abs/2406.06176}{\textsf{arXiv:2406.06176}}},
}
\end{document}